\documentclass{mcom-l}

\usepackage{graphicx}
\usepackage{amsmath}
\usepackage[linesnumbered,ruled,vlined]{algorithm2e}
\usepackage{color}
\usepackage{bm}
\usepackage{amsfonts}
\usepackage{float}
\usepackage{subfig}
\usepackage{epstopdf}
\usepackage{footnote}
\usepackage{tablefootnote}
\usepackage{threeparttable}

\newcommand{\R}{\mathbb{R}}
\newcommand{\tr}{^{\sf T}}
\newcommand{\m}[1]{{\bf{#1}}}

\newtheorem{theorem}{Theorem}[section]
\newtheorem{lemma}[theorem]{Lemma}

\theoremstyle{definition}
\newtheorem{definition}[theorem]{Definition}

\theoremstyle{remark}
\newtheorem{remark}[theorem]{Remark}

\theoremstyle{corollary}
\newtheorem{corollary}[theorem]{Corollary}

\newtheorem{assum}{Assumption}
\numberwithin{equation}{section}

\begin{document}
\title{
A Unified Algorithm for Nonconvex Decentralized  Nonlinear Optimization
}
%
\author{Hao Wu}
\address{School of Mathematics,
	Nanjing University of Aeronautics and Astronautics}
\curraddr{}
\email{wuhoo104@nuaa.edu.cn}
\thanks{}

\author{Liping Wang}
\address{School of Mathematics,
	Nanjing University of Aeronautics and Astronautics}
\curraddr{}
\email{wlpmath@nuaa.edu.cn}
\thanks{}

\subjclass[2020]{Primary }

\date{}

\dedicatory{}

\begin{abstract}
In this paper, we study the decentralized optimization problem of minimizing a finite sum of continuously differentiable and possibly nonconvex functions over a fixed-connected undirected network. 
We propose a unified decentralized nonconvex algorithmic framework that includes many existing state-of-the-art gradient tracking and quasi-Newton algorithms. 
A general framework for the convergence analysis of our unified algorithm is presented 
under both nonconvex and the Kurdyka-{\L}ojasiewicz condition settings.
In particular, some new quasi-Newton algorithms under this framework are proposed. 
 Our numerical results show that these newly developed algorithms are very efficient 
compared with other state-of-the-art algorithms for solving decentralized nonconvex nonlinear optimization.

\end{abstract}

\maketitle

\pagestyle{myheadings}
\thispagestyle{plain}
\markboth{HAO WU, LIPING WANG, AND HONGCHAO ZHANG}
{UNIFIED ALGORITHM FOR NONCONVEX DECENTRALIZED OPTIMIZATION}

	\section{INTRODUCTION}
In the era of data explosion and connected intelligence, decentralized optimization has emerged as a fundamental computational paradigm for large-scale and privacy-aware systems. This paper focuses on solving optimization problems over multi-node networks where no central server exists---a setting that naturally arises in modern applications including but not limited to decentralized resource control \cite{fusco2021decentralized}, wireless networks \cite{jeong2022asynchronous}, decentralized machine learning \cite{zhang2022distributed}, power systems \cite{hussan2024decentralized} and federated learning \cite{lei2025decentralized}. 
More specifically,
we consider a network of $n$ nodes, interconnected via an undirected and connected graph $\mathcal{G}=\left(\mathcal{V},\mathcal{E}\right)$, where $\mathcal{V} = \{1,\ldots,n\}$ is the set of nodes, and $\mathcal{E}$ is the collection of unordered edges. Each node $i \in \mathcal{V}$ possesses a private local objective function $f_i:\R^{p} \rightarrow \R$ and the ultimate goal is to solve the consensus 
optimization problem:
\begin{equation}\label{obj_fun1}
	\mathop {\min } \limits_{\m{z} \in {\R^p}} \; F(\m{z}):= \frac{1}{n}\sum\limits_{i = 1}^n {{f_i}(\m{z})},
\end{equation}
where each $f_i$ is  continuously differentiable, possibly nonconvex, and known only to node $i$. Unlike the distributed setup with a central server, the fully decentralized architecture requires all nodes to collaboratively reach an optimal solution using only local computation and peer-to-peer communication.

For decentralized optimization, gradient-based first-order methods have been widely used
 due to their simple implementation and low computational cost at each iteration.
  Existing well-developed works include Decentralized Gradient Descent (DGD) \cite{nedic2009distributed,yuan2016convergence,zeng2018nonconvex}, gradient tracking methods and
their variants \cite{xu2015augmented,qu2017harnessing,nedic2017achieving,nedic2017geometrically,
jakovetic2019exact,gao2022achieving,zhang2020distributed,song2024optimal}, exact diffusion methods \cite{shi2015extra,yuan2018exact,alghunaim2024local,wu2024music}, momentum methods \cite{xin2019distributed,takezawa2022momentum,huang2024accelerated}, and  primal-dual methods \cite{shi2014linear,jakovetic2020primal,mancino2023decentralized}.
The convergence rates of gradient-based methods are highly sensitive to the condition number 
of the objective function. 
Hence, second-order Newton type decentralized algorithms have been developed to improve the convergence speed
 \cite{mokhtari2016network,mokhtari2016decentralized,zhang2021newton,jakovetic2022hessian,liu2023communication}. Numerical experiments demonstrate that these Hessian-based second-order algorithms exhibit significantly 
faster convergence compared to gradient-based algorithms. However, despite their empirical success, these methods are still theoretically limited to a no-faster-than-linear convergence rate and incur significant computational costs due to the calculation of Hessian matrices and/or their inverses.
To reduce the computational cost, quasi-Newton type decentralized algorithms have also recently been developed
 \cite{eisen2017decentralized,li2021bfgs,zhang2023variance,10589277,wang2025decentralized}, 
where  the curvature information of the Hessian of the objective function is 
captured with a much lower computational cost.
However, most of the existing decentralized quasi-Newton methods are proposed for solving convex or strongly convex minimization problems. \textit{It is still unclear whether some of these methods can be applied to nonconvex decentralized optimization with convergence guarantee.} 
It is well-known that quasi-Newton methods with correction techniques have been extensively used 
in traditional nonlinear optimization to ensure robust performance \cite{li2001modified,liu2014regularized,chen2019stochastic}, where correction techniques are the key to tackling non-convexity and stochasticity.
Recently, under the regularization and damping techniques, Zhang et al.  \cite{zhang2023variance} proposed 
a decentralized quasi-Newton method, which guarantees the positive definiteness of
the updated quasi-Newton matrices even though the overall problem is nonconvex. However, the global convergence still relies on the strong convexity assumption.
The above observations motivate us to develop more in-depth unified analysis on quasi-Newton methods to solve
nonconvex decentralized optimization.

Many works \cite{alghunaim2020decentralized,berahas2024balancing,jakovetic2018unification,sundararajan2020analysis,
xu2021distributed} in the literature have attempted to unify various gradient-based
decentralized optimization methods in the convex or strongly convex settings,
while few works \cite{alghunaim2022unified,du2024unified} provide a unified analysis for solving nonconvex decentralized optimization.
Moreover, the frameworks in \cite{alghunaim2022unified,du2024unified} do not cover quasi-Newton methods.
In this paper, we propose a class of decentralized algorithms for solving nonconvex decentralized optimization,
which also include variants of quasi-Newton methods with guaranteed convergence in the sense that 
every limit point of the iterates is a stationary point.
In recent years, the Kurdyka-{\L}ojasiewicz (K{\L}) property, initially introduced by {\L}ojasiewicz \cite{lojasiewicz1963propriete}, has been successfully employed to strengthen the convergence analysis of many optimization algorithms for nonconvex optimization in the centralized setting \cite{attouch2010proximal,qian2025convergence}. 
However, extending such analyses to decentralized optimization poses significant challenges,
primarily due to the nonmonotone decrease of the objective function value in the convergence analysis.
To facilitate the analysis, most existing works  
\cite{zeng2018nonconvex,daneshmand2020second,wang2024double} on the convergence of decentralized algorithms 
 need to assume a potential function satisfies the K{\L} property and
the iterates generated by a descent algorithm are bounded.
However, the K{\L} property imposed on the potential function cannot generally  transfer to the objective function. 
 Although Chen et al. \cite{chen2024enhancing} leveraged the K{\L} property directly on 
the objective function, they assumes a strong connectivity of the underlying network.
In this paper, under the  K{\L} property of the objective function,
we strengthen the convergence result to show that the entire sequence generated 
by our algorithmic framework converges to a stationary point.
The main contributions of this paper can be summarized as follows. 
\begin{itemize}
	\item [1.] 
	We introduce a Unified Decentralized Optimization Algorithm (UDOA) framework for solving nonconvex decentralized optimization, which includes
various decentralized gradient tracking methods \cite{xu2015augmented,qu2017harnessing,nedic2017achieving,nedic2017geometrically,di2016next,scutari2019distributed} and 
some well-established decentralized quasi-Newton methods \cite{jakovetic2019exact,gao2022achieving,10589277,zhang2023variance}.
This provides a unified approach to analyzing the existing decentralized optimization algorithms. Moreover, based on this framework,
new efficient decentralized algorithms can be designed by choosing different combinations of the parametrization matrices, such as
the  communication and second-order approximation matrices, in this framework.

	\item[2.] We establish the global convergence of UDOA in the nonconvex setting
under standard assumptions, such as Lipschitz continuous gradients and lower-bounded objective values,
which remains challenging for decentralized quasi-Newton methods.
Furthermore, by relying solely on the K{\L} property of the objective function, we show the convergence of the 
entire iterate sequence generated by the algorithm framework, 
addressing the key difficulty of nonmonotone objective function value reductions when establishing convergence in decentralized optimization.
To the best of our knowledge, this is the first time to show the  full sequence convergence of the iterates generated by
a broader class of quasi-Newton decentralized optimization methods.

	\item[3.] 
	Based on the UDOA framework, we propose several innovative and computationally efficient decentralized quasi-Newton methods. 
One key technical contribution is the design of novel strategies within the UDOA framework
to build positive definite quasi-Newton matrices for approximating the Hessian inverse of the objective function.
These strategies only require  vector-vector products which would keep minimum both computation and memory costs per iteration. 
Our numerical experiments show that the newly designed algorithms are highly effective compared with the current state-of-the-art algorithms for 
 decentralized nonconvex optimization.
 
\end{itemize}
The paper is organized as follows. In Section~2, we first propose the UDOA framework and then
show the convergence of the subsequence of the iterates in the Subsections~2.1 and 2.2, respectively.
The convergence of the entire iterate sequence under   
the K{\L} property of the objective function as well as the strategies for generating quasi-Newton approximation matrices 
are presented in Subsections~2.3 and 2.4. 
Numerical experiments on comparing our new methods with other well-established methods for solving decentralized nonconvex optimization are presented in Section~3.
We finally draw some conclusions in Section 4.

\subsection{Notation}
We use uppercase boldface letters, e.g.~$\m{W}$, for matrices and lowercase boldface letters,  e.g.~$\m{w}$, for vectors.
For any vectors $\m{v}_i \in \R^p$, $i=1,\ldots,n$, we define  $\bar{\m{v}}=\frac{1}{n} \sum_{i=1}^n \m{v}_i$ and $\m{v} = [\m{v}_1; \m{v}_2; \ldots; \m{v}_n] \in \R^{np}$.
Given an undirected network  $\mathcal{G}=\left(\mathcal{V},\mathcal{E}\right)$,
let $\m{x}_i \in \R^p $ denote the local copy of the global variable $\m{z} \in \R^p$ at node $i$ and $\mathcal{N}_i$ denote the set consisting of the neighbors of node $i$ 
(for convenience, we treat node $i$ itself as one of its neighbors). 
We define $f(\m{x}) = \sum_{i = 1}^n {{f_i}({\m{x}_i})}$ and use $\m{g}^t$, $\m{g}_i^t$ to stand for $\nabla f(\m{x}^t)$, $\nabla f_i(\m{x}^t_{i})$ respectively,
where, for clarification, the gradient of $f(\m{x})$ is defined as
$ \nabla f(\m{x}) =\left[ \nabla f_1(\m{x}_{1}); \nabla f_2(\m{x}_{2}); \ldots; \nabla f_n(\m{x}_{n}) \right] \in \R^{np}$. 
In addition,  we define $\overline{\nabla} f(\m{x}^t)=\frac{1}{n} \sum_{i=1}^n \nabla f_i(\m{x}_i^t) 
\in \R^p$.
We say that $\m{x}$ is consensual or gets consensus if ${\m{x}_1}={\m{x}_2}=\ldots={\m{x}_n}$. 
Let $\m{I}_p$ be the $p \times p$ identity matrix and $\m{I}$ be $\m{I}_{np}$ for simplicity. The Kronecker product is denoted as $\otimes$.
Given a vector $\m{v}$ and a symmetric  matrix $\m{N}$, $\operatorname{span}(\m{v})$ stands for the linear subspace spanned by $\m{v}$; 
$\mbox{Null}(\m{N})$ denotes the null space of $\m{N}$;
${\lambda _{\min }(\m{N})}$,  ${\lambda _{\max }(\m{N})}$, and $\rho(\m{N})$ denote the smallest eigenvalue, the largest eigenvalue, and  the spectral radius of $\m{N}$,  respectively;
For symmetric matrices $\m{N}_1$ and $\m{N}_2$ with same dimension, $\m{N}_1 \succeq \m{N}_2$ means $\m{N}_1 - \m{N}_2$ is positive semidefinite, while $\m{N}_1 \ge \m{N}_2$ means $\m{N}_1 - \m{N}_2$ is component-wise nonnegative.
We denote $\operatorname{log}_{10}(\cdot)$ by $\operatorname{log}(\cdot)$  and define $\m{M}=\frac{1}{n} \m{1}_n\m{1}_n\tr\otimes \m{I}_p$ where $\m{1}_n \in \R^n$ is the
vector with all components ones. {For any vector $\m{v}$ and matrix $\m{N}$, $\Vert\m{v}\Vert$ and $\Vert\m{N}\Vert$ denote the Euclidean norm of $\m{v}$ and the Frobenius norm of $\m{N}$, respectively.}

\section{ALGORITHM DEVELOPMENT, CONVERGENCE ANALYSIS, AND PRACTICAL IMPLEMENTATION}
We need the following assumptions on the objective function.
\begin{assum}\label{as0} {For the objective function in (\ref{obj_fun1}), we assume the following properties hold.} \\
 (a) The global aggregation function $F(\m{z})= \frac{1}{n}\sum\nolimits_{i = 1}^n {{f_i}(\m{z})}$ is bounded from below, 
i.e., $F(\m{z}) \geq \underline{F}$ for any  $\m{z} \in {\mathbb{R}^p}$, where $\underline{F} \in \R$ is some constant.\\
 (b) For any $i= 1, \ldots, n$, $\nabla f_i$ is Lipschitz continuous with a constant $L>0$, 
i.e., for any $\m{z},\tilde{\m{z}} \in {\mathbb{R}^p}$, it has
	\begin{equation}\label{3.1}
		\left\| {\nabla {f_i}(\m{z}) - \nabla {f_i}(\tilde{\m{z}})} \right\| \le L \left\| {\m{z} - \tilde{\m{z}}} \right\|.
	\end{equation}
\end{assum}

In decentralized optimization it is convenient to parameterize communication by a mixing matrix $\tilde{\m{W}}=[ \tilde{W}_{i j} ]\in \R^{n \times n}$, which is defined as follows.
\begin{definition}\label{mix}
	(Mixing matrix $\tilde{\m{W}}$ for given network $\mathcal{G} = \left(\mathcal{V},\mathcal{E}\right)$)
	\begin{itemize}
		\item [1.] $\tilde{\m{W}}$ is nonnegative, where each component $\tilde{W}_{ij}$ characterizes the active link $(i,j)$, i.e., $\tilde{{W}}_{ij}>0$ if $j \in \mathcal{N}_i$;  $\tilde{{W}}_{ij}=0$, otherwise.
		\item [2.] $\tilde{\m{W}}$ is symmetric and doubly stochastic, i.e., $\tilde{\m{W}}=\tilde{\m{W}}\tr$ and $\tilde{\m{W}}\m{1}_n=\m{1}_n$.\\
	\end{itemize}
\end{definition}
There are a few common choices for the mixing matrix $\tilde{\m{W}}$, such as the Laplacian-based constant edge weight matrix \cite{sayed2014diffusion} and
the Metropolis constant edge weight matrix \cite{xiao2007distributed}.
Let $\lambda_{i}(\tilde{\m{W}})$ denote the $i$-th largest eigenvalue of $\tilde{\m{W}}$ and $\sigma$ be the second largest magnitude eigenvalue of $\tilde{\m{W}}$. 
Then, the following properties hold.
\begin{lemma}\label{property W}
	For $\tilde{\m{W}}$ defined in Definition~\ref{mix} and  $\m{W} :=\tilde{\m{W}}\otimes\m{I}_p$, we have
	\begin{itemize}
		\item [(a)] $1=\lambda_{1}(\tilde{\m{W}})>\lambda_{2}(\tilde{\m{W}})\geq\ldots\geq\lambda_{n}(\tilde{\m{W}})>-1$;
		\item [(b)] $	0<\rho(\m{W}-\m{M})=\sigma=\max \left\{|\lambda_{2}(\tilde{\m{W}})|, |\lambda_{n}(\tilde{\m{W}})|\right\}<1$;
		\item [(c)] $\m{M}=\m{M}\m{W}=\m{W}\m{M} = \m{M}^2$;
		\item [(d)]  
			$\|\m{W}^{k}\m{x}-\m{M}\m{x}\|=\|(\m{W}^{k}-\m{M})(\m{x}-\m{M}\m{x})\| \leq \sigma^{k} \|\m{x}-\m{M}\m{x}\|$
		for any $\m{x} \in \mathbb{R}^{np}$ and $k \geq 1$. 
	\end{itemize}
\end{lemma}
\begin{proof}
	For the proof of properties (a)-(c), one may refer to \cite{xin2019distributed}.
        For property (d), we first observe from property (c) that
\[
(\m{W}^{k}-\m{M})(\m{I} - \m{M}) = \m{W}^{k}-\m{M} -\m{W}^{k}\m{M}+\m{M}^2 = \m{W}^{k}-\m{M}.
\]
Hence, the first equality in property (d) holds. We next show by induction that $(\m{W}-\m{M})^{k}=\m{W}^k-\m{M}$ for $k \geq 1$,
which obviously holds for $k=1$. Now, suppose $(\m{W}-\m{M})^{k-1}=\m{W}^{k-1}-\m{M}$, we have from property (c) that
\begin{align*}
		(\m{W}-\m{M})^{k}=&(\m{W}-\m{M})^{k-1}(\m{W}-\m{M})=(\m{W}^{k-1}-\m{M})(\m{W}-\m{M})\\
		=&\m{W}^k-\m{W}^{k-1}\m{M}-\m{M}\m{W}^{k-1}+\m{M}=\m{W}^k-\m{M}.
\end{align*}
So, we have from property (b) that
\[
\left\|\m{W}^{k}-\m{M} \right\|=\left\|(\m{W}-\m{M})^{k} \right\|\leq \left\|\m{W}-\m{M} \right\|^{k}\leq \sigma^{k},
\]
which gives the second inequality in property (d).
\end{proof}

\subsection{Unified Decentralized Optimization Algorithm}
\begin{algorithm}[htb]
	\caption{Unified Decentralized Optimization Algorithm (UDOA)}
	\label{alg:Framwork}
	
	\SetKwInOut{Input}{Input}
	\SetKwInOut{Output}{Output}
	
	\Input{Initial point $\m{x}^0$, Maximum iteration $T>0$, Stepsize $\alpha>0$, Parameterization matrices $\m{A}=\tilde{\m{A}}\otimes\m{I}_p$, $\m{B}=\tilde{\m{B}}\otimes\m{I}_p$, $\m{C}=\tilde{\m{C}}\otimes\m{I}_p$, and $\m{D}=\tilde{\m{D}}\otimes\m{I}_p$.}
	\Output{$\m{x}^T$.}
	
	\BlankLine
	Set $t = 0$, $\m{v}^0 = \m{g}^0$, and initialize $\m{H}^0$\;
	
	\While{$t < T$}{
		\tcp{Variable update}
		$\m{x}^{t+1} = \m{A}\m{x}^{t}-\alpha \m{B}\m{H}^t \m{v}^{t}$\;
		
		\tcp{Gradient tracking direction update}
		$\m{v}^{t+1} = \m{C}\m{v}^t+\m{D}(\m{g}^{t+1}-\m{g}^t)$\;
		
		\tcp{Metric matrix update}
		Update $\m{H}^{t+1}$\;
		
		$t = t+1$\;
	}
\end{algorithm}
Our unified decentralized optimization algorithm (UDOA) is described in Algorithm~\ref{alg:Framwork},
 where $\alpha >0$ is a stepsize,  $\m{v}^t$ serves as a tracking vector to approximate the global gradient $\nabla F(\bar{\m{x}}^t)$,
$\m{H}^t$ is a symmetric block diagonal matrix whose $i$-th block is $\m{H}_i^t$,
$\m{A}=\tilde{\m{A}}\otimes\m{I}_p$, $\m{B}=\tilde{\m{B}}\otimes\m{I}_p$, $\m{C}=\tilde{\m{C}}\otimes\m{I}_p$, and $\m{D}=\tilde{\m{D}}\otimes\m{I}_p$ are 
some {parametrization} matrices satisfying the following structural assumption.
\begin{assum}\label{mixing}
	{The parametrization matrices $\tilde{\m{A}}$, $\tilde{\m{B}}$, $\tilde{\m{C}}$, and $\tilde{\m{D}}$ are nonnegative, symmetric, and doubly stochastic.
Moreover, $\tilde{\m{A}}$ and $\tilde{\m{C}}$ are not identity matrices for achieving consensus. }
\end{assum}
It is worth noting that the unified framework in \cite{alghunaim2022unified} restricts the {parametrization} matrices to be polynomial functions of 
the mixing matrix $\tilde{\m{W}}$ and must be simultaneously diagonalizable by a same orthogonal matrix for establishing global convergence.
{In addition}, by Assumption~\ref{mixing}, the iteration updates in UDOA, and Lemma~\ref{property W}, 
it can be shown from induction that the gradient tracking vector $\m{v}$ has the property that 
\begin{equation}\label{vequalg}
	\m{M}{\m{v}}^t=\m{M}{\m{g}}^t \quad \Longleftrightarrow \quad \bar{\m{v}}^t=\bar{\m{g}}^t.
\end{equation} 
One may refer to \cite[Lemma 7]{qu2017harnessing} for the detail proof. 
Note that the UDOA framework includes 
not only numerous state-of-the-art first-order decentralized optimization methods but also various decentralized quasi-Newton methods, 
where $\m{H}^t$ typically serves as an approximation to the inverse of the Hessian of the objective function.
We now demonstrate that by different choices of the matrices ${\tilde{\m{A}}, \tilde{\m{B}}, \tilde{\m{C}}, \tilde{\m{D}}, \m{H}^t}$ in
UDOA, we would have different decentralized optimization algorithms.  See Table~\ref{cases}.

\begin{table}
	\centering
	\caption{Special cases in our framework}\label{cases}
			\begin{threeparttable}
\begin{tabular}{cccccc}
	\hline
	\hline
	Methods&$\tilde{\m{A}}$& $\tilde{\m{B}}$& $\tilde{\m{C}}$&  $\tilde{\m{D}}$&  $\m{H}^t$\\
	\hline
	DIGing \cite{nedic2017achieving}, Harnessing \cite{qu2017harnessing}&$\tilde{\m{W}}$  &$\m{I}_p$  &$\tilde{\m{W}}$ &$\m{I}_p$  &$\m{I}$  \\
	SONATA \cite{scutari2019distributed}, NEXT \cite{di2016next}&$\tilde{\m{W}}$  &$\tilde{\m{W}}(\m{I}_p) $ &$\tilde{\m{W}}$  &$\m{I}_p(\tilde{\m{W}}) $ &$\m{I}$  \\
	Aug-DGM \cite{xu2015augmented}, ATC-DIGing \cite{nedic2017geometrically}&$\tilde{\m{W}}$ &$\tilde{\m{W}}$ &$\tilde{\m{W}}$ &$\tilde{\m{W}}$  &$\m{I}$  \\
	DR-LM-DFP \cite{zhang2023variance}, D-LM-BFGS \cite{zhang2023variance}&$\tilde{\m{W}}$  &$\m{I}_p$  &$\tilde{\m{W}}$  &$\m{I}_p$  &BFGS/DFP \tnote{$\ddagger$}  \\
	DQN \cite{shorinwa2024distributed}&$\tilde{\m{W}}$ & $\tilde{\m{W}}^2$ & $\tilde{\m{W}}$ & $\tilde{\m{W}}$ &BFGS  \\
	DGM-BB-C \cite{gao2022achieving}&$\tilde{\m{W}}^K$\tnote{$\dagger$}  &$\tilde{\m{W}}^K$ &$\tilde{\m{W}}^K$ &$\tilde{\m{W}}^K$  &BB  \\
	DSG\cite{jakovetic2019exact} &$\tilde{\m{W}}$ &$\m{I}_p$ &$\tilde{\m{W}}$ &$\m{I}_p$  &\eqref{DSG-update}  \\
	\hline
	\hline
\end{tabular}
\begin{tablenotes}
	\item[$\dagger$]  $\tilde{\m{W}}^K$ represents performing $K$ communications at one iteration.
	\item[$\ddagger$] ``BFGS'', ``DFP'', and ``BB'' means the matrix $\m{H}^t$ is generated by the BFGS, DFP, and BB \cite{barzilai1988two} quasi-Newton updates, respectively.
\end{tablenotes}
		\end{threeparttable}
\end{table}

Let us first consider the relationships of UDOA with numerous gradient-based methods by letting $\m{H}^t=\m{I}$ for all $t \ge 0$.\\
{\bf (I) ATC-GT method.}
Taking $\tilde{\m{A}}=\tilde{\m{B}}=\tilde{\m{C}}=\tilde{\m{D}}=\tilde{\m{W}}$ in Algorithm~\ref{alg:Framwork}, we obtain
 the Adapt-Then-Combine Gradient-Tracking (ATC-GT) method as the following:
\begin{align*} 
	& \m{x}^{t+1}= \m{W}(\m{x}^{t}-\alpha \m{v}^{t}),\\ 
	& \m{v}^{t+1}= \m{W}(\m{v}^t+(\m{g}^{t+1}-\m{g}^t)),
\end{align*}
which {includes} variants of gradient tracking methods such as Aug-DGM \cite{xu2015augmented} and ATC-DIGing \cite{nedic2017geometrically}. \\
{\bf (II) Non-ATC-GT method.}
Notice that the iterations of the DIGing \cite{nedic2017achieving} and Harnessing \cite{qu2017harnessing} methods can be realized by Algorithm~\ref{alg:Framwork}
with $\tilde{\m{A}}=\tilde{\m{C}}=\tilde{\m{W}}$ and $\tilde{\m{B}}=\tilde{\m{D}}=\m{I}_p$, that is
\begin{align*} 
	& \m{x}^{t+1} = \m{W}\m{x}^{t}-\alpha \m{v}^{t},\\ 
	& \m{v}^{t+1} = \m{W}\m{v}^t+(\m{g}^{t+1}-\m{g}^t).
\end{align*}
These methods are not in the ATC form and therefore classified as the non-ATC-GT methods. \\
{\bf (III) Semi-ATC-GT method.} By {taking} $\tilde{\m{A}}=\tilde{\m{B}}=\tilde{\m{C}}=\tilde{\m{W}}$ and $\tilde{\m{D}}=\m{I}_p$ in Algorithm~\ref{alg:Framwork}, we get the updates:
\begin{align*} 
	& \m{x}^{t+1}= \m{W}(\m{x}^{t}-\alpha \m{v}^{t}),\\ 
	& \m{v}^{t+1}= \m{W}\m{v}^t+(\m{g}^{t+1}-\m{g}^t),
\end{align*}
which is equivalent to the recursion:
\begin{align}
	 \m{x}^{t+2} &= \m{W}\m{x}^{t+1}-\alpha\m{W}^2\m{v}^t-\alpha\m{W}(\m{g}^{t+1}-\m{g}^t)\label{semiATCGT}\\
	 & = 2\m{W}\m{x}^{t+1}-\m{W}^2\m{x}^t-\alpha\m{W}(\m{g}^{t+1}-\m{g}^t).\notag
\end{align}
The {update} \eqref{semiATCGT} can be also obtained by eliminating the tracking variable
$\m{v}$ with $\tilde{\m{A}}=\tilde{\m{C}}=\tilde{\m{D}}=\tilde{\m{W}}$ and $\tilde{\m{B}}=\m{I}_p$. Related methods include SONATA \cite{scutari2019distributed} and NEXT \cite{di2016next}, which are classified as the semi-ATC-GT method. 


In the following, we investigate the relationships of UDOA with some existing decentralized quasi-Newton methods  \cite{10589277,zhang2023variance,gao2022achieving}
by setting  $\m{H}^t$ as certain approximation of the Hessian inverse of the objective function. \\
{\bf (IV)  DQN.} 
By setting $\tilde{\m{A}}=\tilde{\m{C}}=\tilde{\m{D}}=\tilde{\m{W}}$ and $\tilde{\m{B}}=\tilde{\m{W}}^2$ in Algorithm~\ref{alg:Framwork},
we recover the DQN method \cite{10589277}, which has the updates:
\begin{align*} 
	& \m{x}^{t+1}= \m{W}(\m{x}^{t}-\alpha \m{d}^{t}), \qquad 
	 \m{v}^{t+1}= \m{W}(\m{v}^t+(\m{g}^{t+1}-\m{g}^t)),\\
	&\m{d}^{t+1}=\m{W}(\m{H}^{t+1}\m{v}^{t+1}),
\end{align*}
where the $i$-th block of $\m{H}^{t+1}$ is represented as
\begin{align}\label{DQN_BFGS}
	\m{H}_i^{t+1}=\left(\m{I}_p-\frac{\m{s}_i^t (\m{y}_i^t)\tr}{(\m{y}_i^t)\tr\m{s}_i^t}\right)\m{H}_i^{t}\left(\m{I}_p-\frac{\m{y}_i^t (\m{s}_i^t)\tr}{(\m{y}_i^t)\tr\m{s}_i^t}\right)+\frac{\m{s}_i^t(\m{s}_i^t)\tr}{(\m{y}_i^t)\tr\m{s}_i^t}
\end{align}
with $\m{s}_i^t=\m{x}_i^{t+1}-\m{x}_i^t$ and $\m{y}_i^t=\m{v}_i^{t+1}-\m{v}_i^t$. \\
{\bf (V) DR-LM-DFP and D-LM-BFGS.} If we choose $\tilde{\m{A}}=\tilde{\m{C}}=\tilde{\m{W}}$ and $\tilde{\m{B}}=\tilde{\m{D}}=\m{I}_p$ in Algorithm~\ref{alg:Framwork}, 
and let $\m{H}^t$ be generated by some damped regularized limited-memory DFP or damped limited-memory BFGS updates, 
we would get the DR-LM-DFP or D-LM-BFGS methods \cite{zhang2023variance}. \\
{\bf (VI) DGM-BB-C.} By setting $\tilde{\m{A}}=\tilde{\m{B}}=\tilde{\m{C}}=\tilde{\m{D}}=\tilde{\m{W}}^K$ in Algorithm~\ref{alg:Framwork},
 we would obtain DGM-BB-C method \cite{gao2022achieving}, which follows the updates:
\begin{align*} 
	& \m{x}^{t+1}= \m{W}^K(\m{x}^{t}-\alpha\m{H}^t\m{v}^{t}),\\ 
	& \m{v}^{t+1}= \m{W}^K(\m{v}^t+(\m{g}^{t+1}-\m{g}^t)),
\end{align*}
where $\m{H}^{t}$ is a block diagonal matrix with the the $i$-th block being updated by applying the Barzilai-Borwein \cite{barzilai1988two} approximation:
\[
\m{H}_i^{t+1}=\frac{\|\m{s}_i^t\|^2}{(\m{s}_i^t)\tr(\m{g}_i^{t+1}-\m{g}_i^{t})}\m{I}_p  \quad \text{or} \quad \frac{(\m{s}_i^t)\tr(\m{g}_i^{t+1}-\m{g}_i^{t})}{\|\m{g}_i^{t+1}-\m{g}_i^{t}\|^2}\m{I}_p.
\]
{\bf (VII) DSG.} By choosing $\tilde{\m{A}}=\tilde{\m{C}}=\tilde{\m{W}}$ and $\tilde{\m{B}}=\tilde{\m{D}}=\m{I}_p$ in Algorithm~\ref{alg:Framwork},
and setting  $\m{H}^{t}$ as a  diagonal matrix with the $i$-th block $\m{H}_i^{t+1}$=$(\delta_i^{t+1})^{-1} \m{I}_p$, where 
\begin{equation}\label{DSG-update}
	\delta_i^{t+1}=\mbox{Proj}_{[\delta_{\min}, \delta_{\max}]}\left\{\frac{(\m{s}_i^t)\tr(\m{g}_i^{t+1}-\m{g}_i^{t})}{\|\m{s}_i^t\|^2}+\delta_i^{t} \sum_{j \in \mathcal{N}_i}\tilde{W}_{ij}\left(1-\frac{(\m{s}_j^t)\tr\m{s}_i^t}{\|\m{s}_i^t\|^2}\right)\right\},
\end{equation}
and $0<\delta_{\min}< \delta_{\max}<\infty$ are parameters,
we would get the DSG method \cite{jakovetic2019exact}.\\[0.05in]

Comparing  UDOA with other unified methods for decentralized optimization, we have the following remarks.
First, notice that Berahas et al. \cite{berahas2024balancing} proposed a Gradient Tracking Algorithmic (GTA) framework which unifies various gradient tracking methods. 
All the methods in the GTA framework are also included in UDOA. However, the convergence of GTA framework
is only studied for the strongly convex objective function. In contrast, we analyze the convergence of UDOA under the more general nonconvex settings.
Second, a stochastic nonconvex unified framework SUDA \cite{alghunaim2022unified} was also proposed, integrating several well-known decentralized methods within a primal-dual formulation.
 While both SUDA and UDOA encompass gradient tracking methods, a key difference between these two methods lies in the integration of quasi-Newton methods. 
By employing a different analyzing approach from SUDA, we establish the global convergence of UDOA with quasi-Newton updates in the nonconvex settings,  
which remains to be a challenging and much less studied area in the literature. Moreover, we are able to establish a stronger convergence result 
(the entire iterate sequence convergence) of UDOA under the  K{\L} properties of the objective function.
Table \ref{Comparisons} summarizes the comparsions of UDOA with several existing unified frameworks, where GT, ED, and QN 
 represent gradient tracking, exact diffusion, and quasi-Newton methods, respectively.
Note that the ABC framework \cite{xu2021distributed}, while built on a similar primal-dual approach to SUDA\cite{alghunaim2022unified}, 
is only restricted to strongly convex objective functions.

	\begin{table}
	\centering
	\caption{Comparisons with existing frameworks}\label{Comparisons}
	\begin{tabular}{ccccc}
		\hline
		\hline
		Frameworks&GTA\cite{berahas2024balancing} &ABC\cite{xu2021distributed}&SUDA \cite{alghunaim2022unified} &UDOA(Ours)\\
		\hline
		Require convexity?& Yes & Yes& No & No\\
		Include GT? &Yes &Yes&Yes &Yes\\
		Include ED? &No &Yes&Yes &No\\
		Include QN? &No &No&No &Yes\\
		Convergence under K{\L}?&Not known &Not known &Not known &Yes\\
		\hline 
		\hline
	\end{tabular}
\end{table}

\subsection{Global Convergence}
We now analyze global convergence of the proposed UDOA, i.e., Algorithm~\ref{alg:Framwork}, for minimizing \eqref{obj_fun1},
 in which the local objective functions $f_i$, $i=1,\ldots,n$,
are Lipschitz continuously differentiable, but possibly nonconvex. For notation convenience, {by Assumption~\ref{mixing}}, we define 
\begin{align*}
	\sigma_{A}=\rho(\m{A}-\m{M}) \in (0,1) \quad \mbox{and} \quad \sigma_{B}=\rho(\m{B}-\m{M}) \in (0,1],\\
	\sigma_{C}=\rho(\m{C}-\m{M}) \in (0,1) \quad \mbox{and} \quad \sigma_{D}=\rho(\m{D}-\m{M}) \in (0,1].
\end{align*}
To establish a unified convergence analysis for all the algorithms included in the UDOA framework, 
we need to require the matrices sequence $\{\m{H}_i^t\}$ to have uniform positive bounded curvature along the direction $\m{v}_i^t$,
regardless of the specific choice of $\m{H}_i^t$. So, we have the following assumption.
\begin{assum}\label{as3}
{The diagonal blocks $\{\m{H}_i^t \}$ of block diagonal matrices $\{\m{H}^t\}$ satisfy}
	\begin{equation*}
		\psi \|\m{v}_i^t\|^2 \leq (\m{v}_i^t)\tr\m{H}_i^t\m{v}_i^t, ~\text{and}~
		\|\m{H}_i^t\m{v}_i^t\| \leq \Psi \|\m{v}_i^t\|,
	\end{equation*}
	for any $t \geq 0$ and $i \in \{1,\ldots,n\}$, where  $\Psi \geq\psi >  0$. 
\end{assum}
{Although standard techniuqes exist in centralized optimization,}
ensuring a sequence  $\{\m{H}_i^t\}$ satisfying Assumption \ref{as3} while approximating the Hessian inverse of the objective function
is a major challenge in nonconvex decentralized optimization.
We will discuss several practical strategies for constructing such a sequence $\{\m{H}_i^t\}$ in Subsection 2.4, 
which are different from those strategies proposed in  \cite{zhang2023variance}.

Recall the notation $\overline{\nabla} f(\m{x}^t)=\frac{1}{n} \sum_{i=1}^n \nabla f_i(\m{x}_i^t)$. 
Then, we obtain $\m{1}_n\otimes\overline{\nabla} f({\m{x}}^t)=\m{M}{\m{g}}^t$. Note that if $\m{x}^* \in \R^{np}$ satisfies 
\begin{equation}\label{stationarity}
	\|\overline{\nabla} f(\m{x}^*)\|^2+\|\m{x}^*-\m{M}\m{x}^*\|^2=0,
\end{equation}
then we have $\m{x}^*_1= \ldots = \m{x}^*_n =: \m{z}^* $ and $\m{z}^* $ would be a first-order stationary point of  \eqref{obj_fun1}.
Since $\m{M}\m{v}^t=\m{M}\m{g}^t$ for any $t$ by \eqref{vequalg}, if
$\|{\m{v}}^t\|\rightarrow 0$, $\|\m{v}^t-\m{M}\m{v}^t\|\rightarrow 0$, and $\|\m{x}^t-\m{M}\m{x}^t\|\rightarrow 0$ as $t \rightarrow \infty$,
any limit of $\m{x}^t$ is a stationary point of  \eqref{obj_fun1} satisfying \eqref{stationarity}.
Hence, for some  $\epsilon >0$, we say $\m{x}^t$ is an $\epsilon$-stationary point of \eqref{obj_fun1} if
\begin{equation}\label{eps-stationary}
	\|{\m{v}}^t\|^2+\|\m{v}^t-\m{M}\m{v}^t\|^2+\|\m{x}^t-\m{M}\m{x}^t\|^2 \leq \epsilon.
\end{equation}
Given any $\epsilon >0$, to show UDOA will generate a $(\m{x}^t, \m{v}^t)$ satisfying \eqref{eps-stationary},
we define the following potential function 
\begin{equation}\label{potential}
	P(\m{x}^{t},\m{v}^{t})=F(\bar{\m{x}}^t)+\|\m{x}^{t}-\m{M}\m{x}^{t}\|^2+\frac{1-(1+\tau)\sigma_{A}^2}{16(1+1/\eta)L^2\sigma_{D}^2}\|\m{v}^{t}-\m{M}\m{v}^{t}\|^2,
\end{equation}
where $\eta>0$ and $\tau>0$ such that $1-(1+\tau)\sigma_{A}^2>0$ are some constants specified in later analysis.
In \eqref{potential}, $F(\bar{\m{x}}^t)$ represents the objective function value at the average of local approximate solutions, 
while $\|\m{x}^{t}-\m{M}\m{x}^{t}\|^2$ measures the consensus error, that is, the deviation of the local approximate solutions from their global average.
The term $\|\m{v}^{t}-\m{M}\m{v}^{t}\|^2$ captures the gradient tracking error, representing the difference between the average-gradient approximations and the true average gradient.
Since $F$ is lower bounded {by $\underline{F}$}, the potential function $P$ in \eqref{potential} is also bounded from below by $\underline{F}$. 
To show a sufficient descent property of $P(\m{x}^{t},\m{v}^{t})$, we begin with deriving a recursive relationship for $F(\bar{\m{x}}^t)$.
\begin{lemma}
Suppose that Assumptions \ref{as0}, \ref{mixing}, and \ref{as3} hold. Let $\{\m{x}^t\}$ be the sequence generated 
by UDOA, i.e., Algorithm~\ref{alg:Framwork}. For all $t \geq 0$, we have
	\begin{align}\label{term1}
		F(\bar{\m{x}}^{t+1})\leq& F(\bar{\m{x}}^t)-\left(\frac{\alpha\psi}{2 n}-\frac{L\alpha^2 \Psi^2}{2n}\right)\|\m{v}^{t}\|^2\\
&+\frac{L^2\alpha\Psi^2}{\psi n}\|\m{x}^t-\m{M}\m{x}^t\|^2+\frac{\alpha\Psi^2}{\psi n}\|{\m{v}}^t-\m{M}{\m{v}}^t\|^2.\notag
	\end{align} 
\end{lemma}

\begin{proof}
	By the $L$-Lipschitz continuity of $\nabla F$, we have
	\begin{align}\label{Fbarx}
		F(\bar{\m{x}}^{t+1})\leq F(\bar{\m{x}}^t)+\Big\langle\nabla F(\bar{\m{x}}^t),\bar{\m{x}}^{t+1}-\bar{\m{x}}^{t}\Big\rangle+\frac{L}{2}\|\bar{\m{x}}^{t+1}-\bar{\m{x}}^{t}\|^{2}.
	\end{align}
	Define an auxiliary sequence $\{\tilde{\m{x}}^t\}_{t=0}^T$ such that for each node $i$, $\tilde{\m{x}}_i^{t+1}={\m{x}}_i^{t}-\alpha \m{H}_i^t\m{v}_i^t$ and $\tilde{\m{x}}_i^{0}={\m{x}}_i^{0}$. Since $\m{x}_i^{t+1}= \sum_{j \in \mathcal{N}_i} \tilde{A}_{ij}\m{x}_j^{t}-\alpha \sum_{j \in \mathcal{N}_i} \tilde{B}_{ij}\m{H}_j^t \m{v}_j^{t}$ and both $\tilde{\m{A}}$ and $\tilde{\m{B}}$ are column stochastic, i.e., $\sum_{i=1}^n \tilde{A}_{ij} =\sum_{i=1}^n \tilde{B}_{ij}= 1, ~ \forall~ j=1,\dots,n$, we obtain that $\frac{1}{n}\sum_{i=1}^n \tilde{\m{x}}_i^{t+1}=\frac{1}{n}\sum_{i=1}^n {\m{x}}_i^{t+1}=\bar{\m{x}}^{t+1}$.
	We now decompose the inner product term in \eqref{Fbarx} as follows:
	\begin{align}\label{terms}
		&\Big\langle\nabla F(\bar{\m{x}}^t),\bar{\m{x}}^{t+1}-\bar{\m{x}}^{t}\Big\rangle
		=\frac{1}{n} \sum_{i=1}^n \Big\langle\nabla F(\bar{\m{x}}^t),\tilde{\m{x}}_i^{t+1}-{\m{x}}_i^{t}\Big\rangle \\
		=&\underbrace{\frac{1}{n} \sum_{i=1}^n \Big\langle{\m{v}}_i^t,\tilde{\m{x}}_i^{t+1}-{\m{x}}_i^{t}\Big\rangle}_{\text {term (I)}}
		+\underbrace{\frac{1}{n} \sum_{i=1}^n \Big\langle\nabla F(\bar{\m{x}}^t)-\overline{\nabla} f({\m{x}}^t),\tilde{\m{x}}_i^{t+1}-{\m{x}}_i^{t}\Big\rangle}_{\text {term (II)}}\notag\\
		&+\underbrace{\frac{1}{n} \sum_{i=1}^n \Big\langle\overline{\nabla} f({\m{x}}^t)-{\m{v}}_i^t,\tilde{\m{x}}_i^{t+1}-{\m{x}}_i^{t}\Big\rangle}_{\text {term (III)}} \notag.
	\end{align}
	In the following, we derive an upper bound for each term in \eqref{terms} separately.\\
	\textbf{Term (I):} By Assumption \ref{as3}, we have
	\begin{align}\label{termA}
		\frac{1}{n} \sum_{i=1}^n \Big\langle{\m{v}}_i^t,\tilde{\m{x}}_i^{t+1}-{\m{x}}_i^{t}\Big\rangle \leq -\frac{\alpha \psi}{ n} \|\m{v}^{t}\|^2.
	\end{align}
	\textbf{Term (II):} Using Young's inequality with some $c>0$ and Assumption \ref{as3}, we deduce 
		\begin{align}\label{termB}
		&\frac{1}{n} \sum_{i=1}^n \Big\langle\nabla F(\bar{\m{x}}^t)-\overline{\nabla} f({\m{x}}^t),\tilde{\m{x}}_i^{t+1}-{\m{x}}_i^{t}\Big\rangle\\
		\leq &\frac{c}{2}\|\nabla F(\bar{\m{x}}^t)-\overline{\nabla} f({\m{x}}^t)\|^2+\frac{\alpha^2\Psi^2}{2nc}\|\m{v}^t\|^2
		\leq \frac{L^2c}{2n}\|\m{M}\m{x}^t-\m{x}^t \|^2+\frac{\alpha^2\Psi^2}{2nc}\|\m{v}^t\|^2,\notag
	\end{align}
	where the last inequality follows from the $L$-Lipschitz continuity of $\nabla f_i$ and the relation $\m{1}_n\otimes\bar{\m{x}}^t=\m{M}{\m{x}}^t$.\\
	\textbf{Term (III):} Applying Young's inequality with some $d>0$, Assumption \ref{as3}, and the relation $\m{1}_n\otimes\overline{\nabla} f({\m{x}}^t)=\m{M}{\m{v}}^t$ from \eqref{vequalg}, we derive
	\begin{align}\label{termC}
		&\frac{1}{n} \sum_{i=1}^n \Big\langle\overline{\nabla} f({\m{x}}^t)-{\m{v}}_i^t,\tilde{\m{x}}_i^{t+1}-{\m{x}}_i^{t}\Big\rangle \\
		\leq &\frac{d}{2n}\|\m{1}_n\otimes\overline{\nabla} f({\m{x}}^t)-{\m{v}}^t\|^2+\frac{\alpha^2\Psi^2}{2nd}\|{\m{v}}^{t}\|^2
		\leq  \frac{d}{2n}\|{\m{v}}^t-\m{M}{\m{v}}^t\|^2+\frac{\alpha^2\Psi^2}{2nd}\|{\m{v}}^{t}\|^2.\notag
	\end{align}
	Plugging \eqref{termA}, \eqref{termB} and \eqref{termC} in \eqref{terms} yields
	\begin{align}\label{middle}
		&\Big\langle\nabla F(\bar{\m{x}}^t),\bar{\m{x}}^{t+1}-\bar{\m{x}}^{t}\Big\rangle\\
		\leq &-\left(\frac{\alpha \psi}{ n}-\frac{\alpha^2 \Psi^2}{2n}\left(\frac{1}{c}+\frac{1}{d}\right)\right)\|\m{v}^{t}\|^2
		+\frac{L^2c}{2n}\|\m{x}^t-\m{M}\m{x}^t\|^2+\frac{d}{2n}\|{\m{v}}^t-\m{M}{\m{v}}^t\|^2.\notag
	\end{align}
For the term $\|\bar{\m{x}}^{t+1}-\bar{\m{x}}^{t}\|^{2}$ in \eqref{Fbarx}, it follows from {Assumption \ref{as3}} that 
\begin{align}\label{third}
	\|\bar{\m{x}}^{t+1}-\bar{\m{x}}^{t}\|^{2}=\frac{1}{n}\|\m{M}(\tilde{\m{x}}^{t+1}-{\m{x}}^{t})\|^{2}\leq \frac{ \alpha^2 \Psi^2}{n}\|\m{v}^t\|^2.
\end{align}
	Finally, substituting \eqref{middle} and \eqref{third} into \eqref{Fbarx}, and choosing $c=2\alpha \Psi^2/ \psi$, $d=2\alpha \Psi^2/ \psi$, 
we obtain the desired inequality \eqref{term1}.
\end{proof}

\begin{theorem}\label{nonconvex_convergence}
	Suppose that Assumptions \ref{as0},  \ref{mixing}, and \ref{as3} hold. Let $\{\m{x}^t\}$ be the sequence generated by UDOA, i.e., Algorithm~\ref{alg:Framwork}. 
If
	\begin{align}\label{alpha}
		\alpha \leq \min \left\{\frac{(1-\sigma_{A}^2)(1-\sigma_{C}^2)^2 \psi n}{256L^2\sigma_{D}^2\Psi^2},\frac{(1-{\sigma}_{A}^2)\psi}{(8L+n+32\sigma_{B}^2 n)\Psi^2},\frac{(1-{\sigma}_{A}^2)\psi n}{8L^2\Psi^2}\right\},
	\end{align}
	then the following convergence rate {holds} for UDOA
	\begin{align}\label{qn-converge-rate}
		\min_{0 \le t \le T} \left\{\|\m{v}^t\|^2+\|{\m{v}}^t-\m{M}{\m{v}}^t\|^2+\|\m{x}^{t}-\m{M}\m{x}^{t}\|^2 \right\}
		\leq \frac{P(\m{x}^{0},\m{v}^{0})-P(\m{x}^{T},\m{v}^{T})}{\gamma \min\{\alpha,1\} T},
	\end{align}
	where $P$ is the potential function defined in \eqref{potential} with $\tau=\frac{1-{\sigma}_{A}^2}{2{\sigma}_{A}^2}$ and $\eta=\frac{1-{\sigma}_{C}^2}{2{\sigma}_{C}^2}$, and $\gamma=\min \{a_1,a_2,a_3\}$ with
	\begin{equation} \label{a1a2a3}
		\left\{
		\begin{array}{cl}
			&a_1=\frac{\psi}{2 n}-\frac{(8L+ n+32\sigma_{B}^2 n)\Psi^2}{4n(1-{\sigma}_A^2)}\alpha \geq \frac{\psi}{4 n}>0,\\
			&a_2=\frac{(1-\sigma_{A}^2)(1-\sigma_{C}^2)^2}{128L^2\sigma_{D}^2}-\frac{\Psi^2\alpha}{\psi n} \geq 
\frac{(1-\sigma_{A}^2)(1-\sigma_{C}^2)^2}{256L^2\sigma_{D}^2} > 0, \\
			&a_3=\frac{1-{\sigma}_{A}^2}{4}-\frac{L^2\Psi^2\alpha}{\psi n} \geq \frac{1-{\sigma}_A^2}{8}>0.
		\end{array}  \right.
	\end{equation}
\end{theorem}
\begin{proof}
First, we establish a recursive upper bound for $\|{\m{x}}^t-\m{M}{\m{x}}^t\|^2$. From Step 3 of Algorithm~\ref{alg:Framwork}, we have
\begin{align}\label{term3}
	&\|\m{x}^{t+1}-\m{M}\m{x}^{t+1}\|^2
	=\|\m{A}\m{x}^t-\m{M}\m{x}^t-\alpha\m{B}\m{H}^t\m{v}^t+\alpha\m{M}\m{H}^t\m{v}^t\|^2\\
	\leq &(1+\tau) \|\m{A}\m{x}^t-\m{M}\m{x}^t\|^2+(1+1/\tau)\alpha^2\|(\m{B}-\m{M})\m{H}^t{\m{v}}^t\|^2\notag\\
	\leq&(1+\tau)\sigma_{A}^2 \|\m{x}^{t}-\m{M}\m{x}^{t}\|^2+(1+1/\tau)\alpha^2 \sigma_{B}^2\Psi^2\|{\m{v}}^t\|^2, \notag 
\end{align}
where the first inequality applies Young's inequality with some $\tau>0$ and the second inequality uses {Lemma \ref{property W}} and {Assumption \ref{as3}}.
	
	Next, we establish a recursive upper bound for $\|{\m{v}}^t-\m{M}{\m{v}}^t\|^2$. From the step 4 of Algorithm~\ref{alg:Framwork}, we have 
	\begin{align*}
		&\|{\m{v}}^{t+1}-\m{M}{\m{v}}^{t+1}\|^2
		=\| \m{C}{\m{v}}^{t}+\m{D}\m{g}^{t+1}-\m{D}\m{g}^{t} -\m{M}\m{v}^t-\m{M}\m{g}^{t+1}+\m{M}\m{g}^t\|^2 \\
		\leq & (1+\eta)\|\m{C}{\m{v}}^t-\m{M}{\m{v}}^t\|^2+(1+1/\eta)\|(\m{D}-\m{M})(\m{g}^{t+1}-\m{g}^t)\|^2\\
		\leq & (1+\eta)\sigma_{C}^2\|{\m{v}}^t-\m{M}{\m{v}}^t\|^2+(1+1/\eta)\sigma_{D}^2L^2\|\m{x}^{t+1}-\m{x}^t\|^2,
			\end{align*}
	where the first inequality follows from Young's inequality with some $\eta>0$ and the second inequality uses {Lemma \ref{property W}} and the $L$-Lipschitz continuity of $\nabla f_i$. 
Note the relation that $\m{x}^{t+1}-\m{x}^t=(\m{A}-\m{I})(\m{x}^t-\m{M}\m{x}^t)-\alpha \m{B}\m{H}^t \m{v}^{t}$ and the spectral properties $\rho(\m{A}-\m{I}) < 2$ and $\rho(\m{A}) \leq 1$ from Assumption \ref{mixing} and Lemma \ref{property W}. 
Then, invoking Lemma \ref{important} and Assumption \ref{as3} yields
		\begin{align}
		&\|{\m{v}}^{t+1}-\m{M}{\m{v}}^{t+1}\|^2
		\leq \|{\m{v}}^t-\m{M}{\m{v}}^t\|^2-(1-(1+\eta)\sigma_{C}^2)\|{\m{v}}^t-\m{M}{\m{v}}^t\|^2 \\
		&+8(1+1/\eta)\sigma_{D}^2  L^2\|\m{x}^{t}-\m{M}\m{x}^{t}\|^2+2(1+1/\eta)\sigma_{D}^2L^2\alpha^2 \Psi^2 \|\m{v}^{t}\|^2. \notag
	\end{align}
	Multiplying both sides of the above inequality by $\frac{1-(1+\tau)\sigma_{A}^2}{16(1+1/\eta)L^2\sigma_{D}^2}$, we have
	\begin{align}\label{term2}
	&\frac{1-(1+\tau)\sigma_{A}^2}{16(1+1/\eta)L^2\sigma_{D}^2}\|{\m{v}}^{t+1}-\m{M}{\m{v}}^{t+1}\|^2\leq \frac{1-(1+\tau)\sigma_{A}^2}{16(1+1/\eta)L^2\sigma_{D}^2}\|{\m{v}}^t-\m{M}{\m{v}}^t\|^2\\
	&-\frac{(1-(1+\tau)\sigma_{A}^2)(1-(1+\eta)\sigma_{C}^2)}{16(1+1/\eta)L^2\sigma_{D}^2}\|{\m{v}}^t-\m{M}{\m{v}}^t\|^2 \notag\\
	&+\frac{1-(1+\tau)\sigma_{A}^2}{2}\|\m{x}^{t}-\m{M}\m{x}^{t}\|^2+\frac{(1-(1+\tau)\sigma_{A}^2)\alpha^2 \Psi^2}{8} \|\m{v}^{t}\|^2. \notag
\end{align}	
 Then, adding \eqref{term1}, \eqref{term2}, and \eqref{term3} yields
	\begin{align*}
		&P(\m{x}^{t+1},\m{v}^{t+1})\leq 	P(\m{x}^{t},\m{v}^{t})-\left(\frac{1-(1+\tau)\sigma_{A}^2}{2}-\frac{L^2\alpha\Psi^2}{\psi n}\right)\|\m{x}^{t}-\m{M}\m{x}^{t}\|^2\\
		&-\left(\frac{\alpha \psi}{2 n}-\frac{L\alpha^2 \Psi^2}{2n}-\frac{(1-(1+\tau)\sigma_{A}^2)\alpha^2 \Psi^2}{8}-\left(1+\frac{1}{\tau}\right)\sigma_{B}^2\alpha^2\Psi^2\right)\|\m{v}^t\|^2\\
		&-\left(\frac{(1-(1+\tau)\sigma_{A}^2)(1-(1+\eta)\sigma_{C}^2)}{16(1+1/\eta)L^2\sigma_{D}^2}-\frac{\alpha\Psi^2}{\psi n}\right)\|{\m{v}}^t-\m{M}{\m{v}}^t\|^2\notag.
	\end{align*}
	Setting $\tau=\frac{1-{\sigma}_{A}^2}{2{\sigma}_{A}^2}$ and $\eta=\frac{1-{\sigma}_{C}^2}{2{\sigma}_{C}^2}$, and  
using the fact that $0 < \sigma_{A} < 1$ and $0 < \sigma_{C} < 1$, we derive 
	\begin{align*}\label{P_descent}
		&P(\m{x}^{t+1},\m{v}^{t+1})
		\leq 	P(\m{x}^{t},\m{v}^{t})-\left(\frac{1-{\sigma}_{A}^2}{4}-\frac{L^2\Psi^2\alpha}{\psi n}\right)\|\m{x}^{t}-\m{M}\m{x}^{t}\|^2
		\notag\\
		&-\left(\frac{\psi}{2n}-\frac{L\Psi^2\alpha}{2n}-\frac{(1-\sigma_{A}^2)\Psi^2\alpha}{16}-\frac{(1+{\sigma}_{A}^2){\sigma}_{B}^2\Psi^2\alpha}{1-{\sigma}_{A}^2}\right)\alpha\|\m{v}^t\|^2\notag\\
		&-\left(\frac{(1-\sigma_{A}^2)(1-\sigma_{C}^2)^2}{64(1+\sigma_{C}^2)L^2\sigma_{D}^2}-\frac{\Psi^2\alpha}{\psi  n}\right)\|{\m{v}}^t-\m{M}{\m{v}}^t\|^2\notag\\
	\leq \; & P(\m{x}^{t},\m{v}^{t}) -\left(\frac{1-{\sigma}_{A}^2}{4}-\frac{L^2\Psi^2\alpha}{\psi n}\right)\|\m{x}^{t}-\m{M}\m{x}^{t}\|^2\notag\\
		&-\left(\frac{\psi}{2 n}-\frac{(8L+ n+32\sigma_{B}^2 n)\Psi^2}{4n(1-{\sigma}_A^2)}\alpha\right)\alpha\|\m{v}^t\|^2\notag\\
		&-\left(\frac{(1-\sigma_{A}^2)(1-\sigma_{C}^2)^2}{128L^2\sigma_{D}^2}-\frac{\Psi^2\alpha}{\psi n}\right)\|{\m{v}}^t-\m{M}{\m{v}}^t\|^2.\notag
	\end{align*}
Hence, by the definitions of $a_1$, $a_2$, and $a_3$ in \eqref{a1a2a3}, we obtain
\begin{equation}\label{P-rec}
	P(\m{x}^{t+1},\m{v}^{t+1})- P(\m{x}^{t},\m{v}^{t}) \le -a_1\alpha\|{\m{v}}^t\|^2
	-a_2\|{\m{v}}^t-\m{M}{\m{v}}^t\|^2
	-a_3\|\m{x}^{t}-\m{M}\m{x}^{t}\|^2,
\end{equation}
where by direct calculation we have 
\begin{align*}
	a_1\geq\frac{\psi}{4 n},\quad a_2 \geq \frac{(1-\sigma_{A}^2)(1-\sigma_{C}^2)^2}{256L^2\sigma_{D}^2}, \quad \mbox{and} \quad a_3 \geq \frac{1-{\sigma}_A^2}{8}.
\end{align*}
Summing \eqref{P-rec} over $t=0, \ldots, T$ and dividing both sides by $\gamma \min\{\alpha,1\}$, we get the desired convergence rate \eqref{qn-converge-rate}, where $\gamma=\min \{a_1,a_2,a_3\}$. 
\end{proof}

{By \eqref{rem_eq0} and the definition of the potential function $P$ given in \eqref{potential},
we have the sequence $\{ P(\m{x}^{t},\m{v}^{t}) \}$ is monotonically nonincreasing and hence, the limit
\begin{equation}\label{def-P-infty}
 P^{\infty} : = \lim_{t \to \infty} P(\m{x}^{t},\m{v}^{t}) \ge \underline{F}
\end{equation}
exits, where $\underline{F}$ is a lower bound of function $F$.}
So, given any $\epsilon >0$, by Theorem~\ref{nonconvex_convergence}, we can see that UDOA will take at most 
\[
(P(\m{x}^{0},\m{v}^{0})- P^{\infty}) / (\epsilon \gamma \min\{\alpha,1\})
\]
iterations  to generate an $\epsilon$-stationary point $\m{x}^t$ satisfying \eqref{eps-stationary}, 
where $\gamma$ is the constant given in \eqref{qn-converge-rate}.
Moreover,  by Theorem~\ref{nonconvex_convergence}, the sequences $\{\|\m{v}^t\|^2\}$, $\{\|\m{x}^{t}-\m{M}\m{x}^{t}\|^2\}$ and $\{\|{\m{v}}^t-\m{M}{\m{v}}^t\|^2\}$ are summable.
Hence, as $t$ goes to infinity, we obtain
\begin{equation}\label{conv_three}
	\|\m{v}^t\|, \quad \|\m{x}^{t}-\m{M}\m{x}^{t}\| \quad \mbox{and} \quad \|{\m{v}}^t-\m{M}{\m{v}}^t\| \rightarrow 0.
\end{equation}
Then, it follows from triangle inequality, 
 $\m{1}_n\otimes\overline{\nabla} f({\m{x}}^t)=\m{M}{\m{v}}^t$, the identity 
$\m{1}_n \otimes \nabla F(\bar{\m{x}}^{t}) =\m{1}_n \otimes \nabla F(\bar{\m{x}}^{t})-\m{1}_n \otimes \overline{\nabla} f({\m{x}}^t)+\m{1}_n \otimes \overline{\nabla} f({\m{x}}^t)-\m{v}^t+\m{v}^t$ and  the $L$-Lipschitz continuity of $\nabla f_i$, we can derive 
\begin{equation}\label{rem_eq0}
	\|\nabla F(\bar{\m{x}}^t)\| \rightarrow 0.
\end{equation}
Furthermore, by  \eqref{conv_three}, we can also derive 
\begin{equation}\label{rem_eq}
 	 F(\bar{\m{x}}^t) \rightarrow P^{\infty}.
\end{equation}

\begin{remark}
	{Theorem~\ref{nonconvex_convergence} implies any limit point of $\{\m{x}^t\}$ generated by UDOA would be a stationary point of  the problem \eqref{obj_fun1}.
In fact, if  $\m{x}^{\infty}$ is a limit point, i.e., exists a subsequence, denoted by $\{\m{x}^{t_k}\}_{k \in \mathbb{N}} \rightarrow \m{x}^{\infty}$ 
as $k\rightarrow\infty$. Then, from \eqref{rem_eq0}, we would have $\|\nabla F(\bar{\m{x}}^{t_k})\|+\|\m{x}^{t_k}-\m{M}\m{x}^{t_k}\| \rightarrow 0$ as $k\rightarrow\infty$,
 which shows $\|\nabla F(\bar{\m{x}}^{\infty})\|+\|\m{x}^{\infty}-\m{M}\m{x}^{\infty}\| = 0$ and $\m{x}^{\infty}$ is a stationary point.
Moreover, by \eqref{P-rec} and the definition of the potential function $P$ given in \eqref{potential}, we have
$F(\bar{\m{x}}^t)  \leq P(\m{x}^t,\m{v}^t)  \leq P(\m{x}^{0},\m{v}^{0})$.
Hence, if $F$ is  coercive, i.e.,  $F(\m{z})\rightarrow +\infty$ as $\|\m{z}\| \rightarrow \infty$, 
we have $\{\bar{\m{x}}^t\}$ is bounded, which also implies $\{\m{x}^t\}$ is bounded.
 Hence, in this case,  $\{\m{x}^t\}$ has at least one accumulation point and any accumulation point 
is a stationary point of the problem \eqref{obj_fun1}. 
 }
\end{remark}

Although any limit point of $\{\m{x}^t\}$ is  a stationary point of  the problem \eqref{obj_fun1} by Theorem~\ref{nonconvex_convergence},
the convergence of the entire sequence  $\{\m{x}^t\}$ is not guaranteed. 
We would strengthen the convergence result such that the entire sequence  $\{\m{x}^t\}$ converges under the condition that 
$F$ has the K{\L} property.

\subsection{Convergence under K{\L} Property}
We now show the entire sequence $\{\m{x}^t\}$ converges under the assumption that  $F$ has the K{\L} property,
where the associated convergence rate would naturally emerge.   
\begin{definition}\label{KL1}
	(K{\L} property $\&$ K{\L} function) A proper closed function $h$ is said to have the Kurdyka-{\L}ojasiewicz (K{\L}) property at $\m{z} \in \operatorname{dom} h$ if there exist a neighborhood $\mathcal{N}$ of $\m{z}$, $v \in (0,\infty]$ and a continuous concave function $\phi:[0,v) \rightarrow \mathbb{R}_{\geq 0}$ with $\phi(0)=0$ such that the following holds.
	\begin{itemize}
		\item [1.] $\phi$ is continuously differentiable on $(0, v)$ with $\phi'>0$.
		\item[2.] For all $\tilde{\m{z}} \in \mathcal{N}$ with $0< |h(\tilde{\m{z}})-h(\m{z})| <v$, it holds that
		$$
		\phi'(h(\tilde{\m{z}} )-h(\m{z}))\operatorname{dist}(0,\partial h(\tilde{\m{z}} )) \geq 1.
		$$
	\end{itemize}
	A proper closed function $h$ satisfying the K{\L} property at all points in $\operatorname{dom} h$ is called a K{\L} function.
\end{definition}
Of particular interest in our analysis is the notion of the K{\L} exponent.
\begin{definition}\label{KL2}
	(K{\L} exponent) For a proper closed function $h$ satisfying the K{\L} property at $\m{z} \in \operatorname{dom} h$, if  the desingularizing  function $\phi$ can be chosen as $\phi(s)=c s^{1-\theta}$ for some $c>0$ and $\theta \in [0,1)$, i.e., there exists $\epsilon>0$ and $v \in (0,\infty]$ such that
	\begin{equation}\label{KLeq}
		\operatorname{dist}(0,\partial h(\tilde{\m{z}})) \geq \kappa|h(\tilde{\m{z}})-h(\m{z})|^{\theta}
	\end{equation}
	whenever $\|\tilde{\m{z}}-\m{z}\| \leq \epsilon$ and $0< |h(\tilde{\m{z}})-h(\m{z})| <v$, where $\kappa=\frac{1}{c(1-\theta)}$, then $h$ is said to have the K{\L} property at $\m{z}$ with exponent $\theta$. If $h$ is a K{\L} function and has the same exponent $\theta$ in $\operatorname{dom} h$, then we say that $h$ is
	a K{\L} function with exponent $\theta$.
\end{definition}
\begin{remark}
	{In literature \cite{attouch2010proximal,li2018calculus}, the K{\L} property is typically defined for $h(\tilde{\m{z}}) > h(\m{z})$.
 However, in decentralized optimization, while the potential function  $P(\m{x}^t,\m{v}^t) $ strictly decreases, 
the objective function value $F(\bar{\m{x}}^t)$ may oscillate and could be strictly smaller than its limit $P^\infty$. 
Therefore, our analysis utilizes Definition~\ref{KL2} to accommodate the nonmonotone reduction of $F(\bar{\m{x}}^t)$ in decentralized optimization.
	It is worth noting that this inequality \eqref{KLeq} is naturally satisfied by semi-algebraic functions (which encompass the objective functions evaluated in our numerical experiments) around their critical points.}
\end{remark}

{In this subsection, we also use the following notation to facilitate our analysis:
\begin{align}\label{Newn}
	\m{X}=[\m{x}_{1} \tr ;\m{x}_{2} \tr; \ldots ;\m{x}_{n} \tr ] \in \mathbb{R}^{n\times p},~
	\m{V}=[\m{v}_{1} \tr ;\m{v}_{2} \tr; \ldots ;\m{v}_{n} \tr ]\in \mathbb{R}^{n\times p},~
	\tilde{\m{M}}=\frac{1}{n} \m{1}_n \m{1}_n \tr.
\end{align} 
Furthermore, to evaluate the global objective function using matrix variable, we define the function $G: \mathbb{R}^{n \times p} \rightarrow \mathbb{R}$ as
\begin{equation}\label{def-G}
	G(\m{X}) := F\left(\frac{1}{n}\m{X}\tr\m{1}_n\right) = F(\bar{\m{x}}),
\end{equation} 
and conveniently define $G(\m{X}^t) = F(\bar{\m{x}}^t)$. By the chain rule, the gradient of $G(\m{X})$ with respect to $\m{X}$ is given by $\nabla G(\m{X}) = \frac{1}{n}\m{1}_n \nabla F(\bar{\m{x}})\tr \in \mathbb{R}^{n \times p}$. Moreover, in the following analysis,  we would frequently use the inequality 
\begin{equation*}
	\|\m{N}_1\m{N}_2\| \leq \|\m{N}_1\|_2 \|\m{N}_2\|
\end{equation*}
for any $\m{N}_1 \in \R^{n \times n}$ and $\m{N}_2 \in \R^{n \times p}$, where $\|\cdot\|_2$ is the matrix 2-norm. }

The following lemma shows that the K{\L} property of $F$ naturally transfers to $G$.
\begin{lemma}\label{KL-G}
	Let $\m{X}^o=\m{1}_n(\m{z}^o)\tr$ for $\m{z}^o \in \R^p$. If $F$ has the K{\L} property at $\m{z}^o$ with exponent $\theta_F \in [0,1)$ and parameters $(\kappa_F >0, \epsilon_F>0, v_F>0)$, then $G$ inherits the K{\L} property at ${\m{X}^o}$ with exponent $\theta_G=\theta_F$ and the parameters 
$(\kappa_G=\frac{\kappa_F}{\sqrt{n}},\epsilon_G=\sqrt{n}\epsilon,v_G=v_F)$. 
\end{lemma}
\begin{proof}
	By the K{\L} property of $F$, there exists a neighborhood $V=\{\m{z}\in \mathbb{R}^{p}: \|\m{z}-\m{z}^o\| \leq \epsilon_F\}$ of ${\m{z}^o}$, a constant $\kappa_F>0$, and an exponent $\theta_F \in [0,1)$ such that for all $\m{z} \in V \cap \{\m{z}:0<|F(\m{z}) - F(\m{z}^o)|<v_F\}$, the following inequality holds
	\begin{equation} \label{KL-F}
		\|\nabla F(\m{z}) \| \geq \kappa_F|F(\m{z})-F(\m{z}^o)|^{\theta_F}.
	\end{equation}
	We define a neighborhood $U$ of ${\m{X}^o}$:
	\begin{align*}
		U=\{\m{X}\in \mathbb{R}^{n\times p}: \|\m{X}-{\m{X}^o}\| \leq \epsilon_G\},
	\end{align*}
where  $\epsilon_G=\sqrt{n} \epsilon_F$. 
Then, for any $\m{X} \in U$, we have $\|\bar{\m{x}}-\m{z}^o\| \leq \frac{1}{\sqrt{n}}\|\m{X}-{\m{X}}^o\| \le  \epsilon_F$,
where $\bar{\m{x}}=\frac{1}{n} \m{X}\tr\m{1}_n $.
Therefore, for any $\m{X} \in U \cap \{\m{X}:0<|G(\m{X}) - G(\m{X}^o)|<v_F\}$, we have from \eqref{KL-F},
\eqref{def-G} and $\|\nabla G(\m{X})\|=\frac{1}{\sqrt{n}} \|\nabla F(\bar{\m{x}})\|$ that
	\begin{align*}
		\sqrt{n}\|\nabla G(\m{X}) \| \geq \kappa_F|G(\m{X})-G({\m{X}}^o)|^{\theta_F},
	\end{align*}
	which gives the final K{\L} inequality for $G$:
	\begin{align*}
		\|\nabla G(\m{X}) \| \geq \frac{\kappa_F}{\sqrt{n}}|G(\m{X})-G({\m{X}}^o)|^{\theta_F}.
	\end{align*}
	This indicates $G$ satisfies the K{\L} property at ${\m{X}}^o$ with exponent $\theta_G = \theta_F$ and the parameters $(\kappa_G=\frac{\kappa_F}{\sqrt{n}},\epsilon_G=\sqrt{n}\epsilon_F,v_G=v_F)$. 
\end{proof}
{In the following, let us define}
\begin{align}
	&P^t=P(\m{x}^{t},\m{v}^{t}), \quad \Delta P^t=P^t-P^{\infty} \quad \mbox{and} \label{not1} \\
	& \mathcal{T}^t=\sqrt{\gamma\min\{1,\alpha\}}\left(\|\m{V}^t\|^2+\|{\m{V}}^t-\tilde{\m{M}}{\m{V}}^t\|^2+\|\m{X}^{t}-\tilde{\m{M}}\m{X}^{t}\|^2\right)^{1/2},\label{not2}
\end{align}
where $\gamma$ is the constant from Theorem \ref{nonconvex_convergence} and $P^{\infty}$ is the 
limit of $\{P^t\}$ defined in \eqref{def-P-infty}.
Then, it can be readily verified from \eqref{P-rec} that
\begin{equation}\label{Tt}
	(\mathcal{T}^t)^2 \leq \Delta P^t-\Delta P^{t+1}.
\end{equation}
To prove the convergence of the entire sequence $\{{\m{X}}^t\}$ generated by UDOA, we will show that the distance sequence $\{\|{\m{X}}^{t+1}-{\m{X}}^t\|\}$ is summable, i.e., $\sum_{t=0}^{\infty}\|{\m{X}}^{t+1}-{\m{X}}^t\| < \infty$, {which directly implies the convergence of  $\{{\m{X}}^t\}$.}
The following lemma first reveals that $\|{\m{X}}^{t+1}-{\m{X}}^t\|$ is upper bounded by $\mathcal{T}^t$.
\begin{lemma}\label{distance_control}
	Under the same assumptions and parameter choices of {Theorem \ref{nonconvex_convergence}}, it holds that
	\begin{align}\label{xt+1-xt}
		\|{\m{X}}^{t+1}-{\m{X}}^t\|
		\leq \sqrt{c_1} \mathcal{T}^t,
	\end{align}
where $c_1=\frac{2\max\{\alpha^2\Psi^2,4\}}{\gamma\min\{1,\alpha\}}$.
\end{lemma}
\begin{proof}
	Using the notation in \eqref{Newn}, it is convenient to rewrite the recursion of $\m{X}^t$ as 
	\begin{equation*}
		{\m{X}}^{t+1}=\tilde{\m{A}}{\m{X}}^{t}-\alpha\tilde{\m{B}}\m{D}^t,
	\end{equation*}
where $\m{D}^t=[(\m{d}_{1}^t) \tr;(\m{d}_{2}^t) \tr;\dots;(\m{d}_{n}^t) \tr]$ with $\m{d}_i^t=\m{H}_i^t \m{v}_i^t$, $i=1,\ldots,n$.
Observe the relation that ${\m{X}}^{t+1}-{\m{X}}^t=(\tilde{\m{A}}-\m{I}_n)(\m{X}^t-\tilde{\m{M}}\m{X}^t)-\alpha\tilde{\m{B}}\m{D}^t$ and the spectral properties that $\rho(\tilde{\m{A}}-\m{I}_n) <2$ and $\rho(\tilde{\m{B}}) \leq 1$ {by} Assumption \ref{mixing} and Lemma \ref{property W}.
Then it follows from Lemma \ref{important} and Assumption \ref{as3} that
	\begin{align*}
		\|{\m{X}}^{t+1}-{\m{X}}^t\|
		\leq\sqrt{8\|\m{X}^t-\tilde{\m{M}}\m{X}^t\|^2+2\alpha^2\Psi^2\|\m{V}^t\|^2}
		\leq \sqrt{c_1} \mathcal{T}^t,
	\end{align*}
	where  $c_1=\frac{2\max\{\alpha^2\Psi^2,4\}}{\gamma\min\{1,\alpha\}}$.
\end{proof}
Therefore, by {Lemma \ref{distance_control}}, it suffices to show that the sequence $\{{\mathcal{T}^t}\}$ is summable. 
In the following lemma, we establish a key relationship between $P^{t+1}$ and $G(\m{X}^t)$, and derive an upper bound of $\|\nabla G(\m{X}^{t})\|$ in terms of $\mathcal{T}^t$, 
\begin{lemma}\label{two_ineq}
	Under the same assumptions and parameter choices as in {Theorem \ref{nonconvex_convergence}}, the following two inequalities hold:
\begin{align}\label{DeltaP}
	\Delta P^{t+1}
	\leq G(\m{X}^t)+ c_2 (\mathcal{T}^t)^2 -P^{\infty}
\end{align}
and
	\begin{align}\label{F-P2}
	\|\nabla G(\m{X}^t)\|
	\leq  c_3\mathcal{T}^t,
\end{align}
where $P^{\infty}$ is given in \eqref{rem_eq}, 
$c_2=\max\left\{1,\frac{(1-\sigma_{A}^2)(1-\sigma_{C}^2)}{32(1+\sigma_{C}^2)L^2\sigma_{D}^2}\right\}/(\gamma\min\{1,\alpha\})$, and $c_3=\sqrt{\frac{3\max\{1,L^2\}}{\gamma\min\{1,\alpha\}n^2}}$.
\end{lemma}
\begin{proof}
{By the descent property of $P^{t}$ and the definitions of
$P^t$ and  $\mathcal{T}^t$ in \eqref{not1} and \eqref{not2}, it has } 
\begin{align*}
	&\Delta P^{t+1} \leq P^{t}-P^{\infty}\\
	\leq \; & F((1/n)(\m{X}^t)\tr\m{1}_n)+\|\m{X}^{t}-\tilde{\m{M}}\m{X}^{t}\|^2+\frac{(1-\sigma_{A}^2)(1-\sigma_{C}^2)}{32(1+\sigma_{C}^2)L^2\sigma_{D}^2}\|\m{V}^{t}-\tilde{\m{M}}\m{V}^{t}\|^2-P^{\infty}\\
	\leq \; & G(\m{X}^t)+ c_2 (\mathcal{T}^t)^2 -P^{\infty},
\end{align*}
where $c_2=\max\left\{1,\frac{(1-\sigma_{A}^2)(1-\sigma_{C}^2)}{32(1+\sigma_{C}^2)L^2\sigma_{D}^2}\right\}/(\gamma\min\{1,\alpha\})$. 
So, the inequality \eqref{DeltaP} holds.
	Note  that $\nabla G(\m{X}^t)= (1/n)\m{1}_n \nabla F(\bar{\m{x}}^{t})\tr$ and we have the decomposition that 
	$$
	\m{1}_n\nabla F(\bar{\m{x}}^{t})\tr =\m{1}_n\nabla F(\bar{\m{x}}^{t})\tr-\m{1}_n\overline{\nabla} f({\m{x}}^t)\tr+\m{1}_n\overline{\nabla} f({\m{x}}^t)\tr-\m{V}^t+\m{V}^t.
	$$ 
	Then, by {Lemma \ref{important}}, the $L$-Lipschitz continuity of each $\nabla f_i$, and the relation that $\m{1}_n\overline{\nabla} f({\m{x}}^t)\tr=\tilde{\m{M}}\m{V}^t$, we have
	\begin{align*}
		&\|\nabla G(\m{X}^t)\|^2=\frac{1}{n}\|\nabla F(\bar{\m{x}}^{t})\|^2 \\
		\leq& \frac{3L^2}{n^2}\|\m{X}^t-\tilde{\m{M}}\m{X}^t\|^2+\frac{3}{n^2}\|\m{V}^t-\tilde{\m{M}}\m{V}^t\|^2+\frac{3}{n^2}\|\m{V}^t\|^2
		\leq  (c_3\mathcal{T}^t)^2,
	\end{align*}
	where $c_3=\sqrt{\frac{3\max\{1,L^2\}}{\gamma\min\{1,\alpha\}n^2}}$, which gives the inequality \eqref{F-P2}.
\end{proof}

Assuming that $F$ satisfies the K{\L} property, Lemma~\ref{KL-G} shows that $G$ would inherit this property. 
Then, from inequality~\eqref{F-P2}, we can bound $G(\m{X}^t)-P^{\infty}$ in terms of $\mathcal{T}^t$. 
These results together with the relation given in~\eqref{DeltaP} would lead to an upper bound of 
$\Delta P^{t+1}$ in terms $\mathcal{T}^t$ in the following lemma.
\begin{lemma}\label{lemma 2.8}
	Under the same assumptions and parameter choices as in {Theorem \ref{nonconvex_convergence}}, let $\m{X}^{\infty}=\m{1}_n(\bar{\m{x}}^{\infty})\tr$ be an accumulation point of $\m{X}^t$, where $\bar{\m{x}}^{\infty}$ is some critical point of $F$. Assume $F$ satisfies the K{\L} property at $\bar{\m{x}}^{\infty}$ with exponent $\theta_F=\theta \in [0,1)$ and the  parameters $(\kappa_F=\sqrt{n}\kappa >0,\epsilon_F>0,v_F=1) $. Then, there exists a neighborhood $\mathcal{N}_{\infty}$ of ${\m{X}}^{\infty}$ and {an integer}
 $t_1 \ge 0$ such that for all $t \in \mathcal{S}:=\{t\geq t_1:{\m{X}}^{t} \in \mathcal{N}_{\infty} \} \neq \emptyset$, we have
	$$
	\max\{c_3,1\}\mathcal{T}^t<1 \quad \text{and} \quad |G({\m{X}}^t)-P^{\infty}|<1,
	$$
	and  for $\theta \in (0,1)$ and $G({\m{X}}^t)-P^{\infty}\neq 0$, we have
\begin{equation}\label{DeltaPt+1}
	\Delta P^{t+1} \leq\kappa^{-\frac{1}{\theta}}(c_3\mathcal{T}^t)^{\frac{1}{\theta}}+ c_2 (\mathcal{T}^t)^2,
\end{equation}
	where $c_2=\max\left\{1,\frac{(1-\sigma_{A}^2)(1-\sigma_{C}^2)}{32(1+\sigma_{C}^2)L^2\sigma_{D}^2}\right\}/(\gamma\min\{1,\alpha\})$ and $c_3=\sqrt{\frac{3\max\{1,L^2\}}{\gamma\min\{1,\alpha\}n^2}}$.
\end{lemma}
\begin{proof}
	Since ${\m{X}}^{\infty}$ is an accumulation point of $\{\m{X}^t\}$, we have $P^{\infty}=G(\m{X}^{\infty})$. By the K{\L} assumption on $F$ and {Lemma \ref{KL-G}}, we conclude that $G$ possesses the K{\L} property at ${\m{X}}^{\infty}$ with exponent $\theta_G=\theta \in [0,1)$ and parameters $(\kappa_G=\kappa,\epsilon_G=\sqrt{n}\epsilon_F>0, v_G=1) $. Hence, there exists a neighborhood $\mathcal{N}_{\infty}$ of ${\m{X}}^{\infty}$ such that for all $\m{X} \in \mathcal{N}_{\infty} \cap \{\m{X}:0<|G(\m{X}) - P^{\infty}|<1\}$, we have
	\begin{equation}\label{F-P}
		|G(\m{X})-P^{\infty}|^{\theta} \leq \frac{\|\nabla G(\m{X})\|}{\kappa}.
	\end{equation}
From {Theorem \ref{nonconvex_convergence}}, we know that $\mathcal{T}^t \rightarrow 0$ and $G({\m{X}}^{t}) \rightarrow P^{\infty}$ as $t \to \infty$ . Thus, there exists some integer 
$t_1 \ge 0$ such that for all $ t\geq t_1$, we have
$$
\max\{c_3,1\}\mathcal{T}^t<1 \quad \text{and} \quad |G({\m{X}}^t)-P^{\infty}|<1,
$$
where $c_3=\sqrt{\frac{3\max\{1,L^2\}}{\gamma\min\{1,\alpha\}n^2}}$. 
Define $\mathcal{S}=\{t\geq t_1:{\m{X}}^{t} \in \mathcal{N}_{\infty} \}$, which is nonempty since ${\m{X}}^{\infty}$ is an accumulation point of $\{\m{X}^t\}$. For $\theta \in (0,1)$ and any $t \in \mathcal{S}$ with $G({\m{X}}^t)-P^{\infty}\neq 0$, the inequality \eqref{F-P} implies
\begin{equation}\label{F-P1}
	|G({\m{X}}^{t})-P^{\infty}| \leq \left(\frac{\|\nabla G({\m{X}}^{t})\|}{\kappa}\right)^{1/\theta}.
\end{equation}
 Substituting \eqref{F-P2} into \eqref{F-P1} yields 	\begin{equation}\label{G-P}
	|G({\m{X}}^{t})-P^{\infty} |\leq \kappa^{-\frac{1}{\theta}}(c_3\mathcal{T}^t)^{\frac{1}{\theta}}.
\end{equation}
Combining \eqref{G-P} with \eqref{DeltaP} gives the desired bound \eqref{DeltaPt+1}.
\end{proof}
In the setting of {Lemma \ref{lemma 2.8}}, we may assume, without loss of generality, that $|\mathcal{S}|=\infty$.
Using \eqref{Tt} and {Lemma \ref{important2}}  with $a_1=\Delta P^t$ and $a_2=\Delta P^{t+1}$ gives
\begin{equation}\label{Tt2}
	(\mathcal{T}^t)^2 \leq \Delta P^t-\Delta P^{t+1} \leq  \frac{1}{1-\theta} (\Delta P^t)^{\theta} [\underbrace{(\Delta P^t)^{1-\theta}-(\Delta P^{t+1})^{1-\theta}}_{\Delta P^t_{\theta}}].
\end{equation}
Based on the above preliminaries, we now establish the summability of $\{\mathcal{T}^t\}$.
\begin{theorem}\label{whole-convergence}
	Suppose that Assumptions \ref{as0}, \ref{mixing}, and \ref{as3} hold. Let $\{\m{X}^t\}$ be the sequence generated by UDOA, i.e., Algorithm~\ref{alg:Framwork}, with the stepsize
		\begin{align*}
	\alpha \leq \min \left\{\frac{(1-\sigma_{A}^2)(1-\sigma_{C}^2)^2 \psi n}{256L^2\sigma_{D}^2\Psi^2},\frac{(1-{\sigma}_{A}^2)\psi}{(8L+n+32\sigma_{B}^2 n)\Psi^2},\frac{(1-{\sigma}_{A}^2)\psi n}{8L^2\Psi^2}\right\}.
	\end{align*}
	Let $\m{X}^{\infty}=\m{1}_n(\bar{\m{x}}^{\infty})\tr$ be an accumulation point of $\m{X}^t$, where $\bar{\m{x}}^{\infty}$ is some critical point of $F$. If $F$ satisfies the K{\L} property at $\bar{\m{x}}^{\infty}$ with exponent $\theta_F=\theta \in [0,1)$ and the parameters $(\kappa_F=\sqrt{n}\kappa >0,\epsilon_F>0,v_F=1) $, then $\{\mathcal{T}^t\}$ is summable and the entire sequence $\{{\m{X}}^t\}$ converges.
\end{theorem}
\begin{proof}
	By {Lemma \ref{distance_control}}, $\|{\m{X}}^{t+1}-{\m{X}}^t\|
	\leq \sqrt{c_1} \mathcal{T}^t$. Thus, the summability of 
	$\{\mathcal{T}^t\}$ implies $\{{\m{X}}^t\}$ is a Cauchy sequence and hence convergent.
  First, notice that if  $G({\m{X}}^t) = P^\infty$, \eqref{DeltaP}
   and $\kappa^{-\frac{1}{\theta}}(c_3\mathcal{T}^t)^{\frac{1}{\theta}} >0 $ would also imply 
    \eqref{DeltaPt+1} holds. Hence, the upper bound in \eqref{DeltaPt+1} holds 
 for all $t \in \mathcal{S}$, regardless of whether $G({\m{X}}^t) = P^\infty$ or not. 
	
 In the following, we prove the summability of $\{\mathcal{T}^t\}$ by considering three cases based on the K{\L} exponent $\theta$.\\
	\textbf{Case I}: $\theta \in (0,1/2]$. By Lemma \ref{lemma 2.8}, $c_3\mathcal{T}^t<1$ for $t \in \mathcal{S}$. Since $1/\theta \geq 2$, inequality \eqref{DeltaPt+1} implies
	\begin{equation}\label{DeltaPt2}
		\Delta P^{t+1} \leq c_4 (\mathcal{T}^t)^2,
	\end{equation}
where $c_4=\kappa^{-\frac{1}{\theta}}c_3^2+c_2$. 
Multiplying \eqref{DeltaPt2} by $\omega >0 $ and adding it to \eqref{Tt} gives
\begin{equation*}
	(1+\omega)\Delta P^{t+1} \leq \Delta P^{t}-(1-\omega c_4)(\mathcal{T}^t)^2.
\end{equation*}
Choosing $\omega \in (0,1/c_4)$ ensures that
\begin{equation}\label{DeltaPt3}
	\Delta P^{t+1} \leq \frac{1}{1+\omega} \Delta P^{t} \leq \Delta P^{0} \left(\frac{1}{1+\omega}\right)^{t+1}.
\end{equation}
Then, applying \eqref{DeltaPt3} to \eqref{Tt} yields
\begin{equation}\label{case1}
	\mathcal{T}^t \leq \sqrt{\Delta P^{0}} (\tau_1)^t,
\end{equation}
where $\tau_1=1/\sqrt{1+\omega}$.\\
\textbf{Case II}: $\theta \in (1/2,1)$. From \eqref{DeltaPt+1} and the relation that $\mathcal{T}^t<1$ for $t \in \mathcal{S}$ derived by Lemma \ref{lemma 2.8}, we have
\begin{equation}\label{DeltaPt4}
	\Delta P^{t+1} \leq c_5 (\mathcal{T}^t)^{1/\theta},
\end{equation}
where $c_5=(c_3/\kappa)^{1/\theta}+c_2$. Using \eqref{Tt2} and \eqref{DeltaPt4}, 
we obtain
\begin{align}\label{Tt2+}
	\mathcal{T}^{t+1} \leq \sqrt{\frac{1}{1-\theta} (\Delta P^{t+1})^{\theta} \Delta P^{t+1}_{\theta}}
	\leq \sqrt{\frac{c_5^{\theta} }{1-\theta} \mathcal{T}^t \Delta P^{t+1}_{\theta}}\leq\frac{1}{2}\mathcal{T}^{t}+\frac{c_5^{\theta} }{2-2\theta} \Delta P^{t+1}_{\theta},
\end{align}
where $\Delta P^{t+1}_{\theta}$ is defined in \eqref{Tt2}.
\textbf{Case III}: $\theta=0$. For any ${t} \in  \mathcal{S}$, the K{\L} property of $G$ at ${\m{X}}^{\infty}$ implies
\begin{equation}
	\kappa(G({\m{X}}^{t}) -P^{\infty})^0 \leq \|\nabla G({\m{X}}^{t})\|.
\end{equation}
Combining this with $\|\nabla G({\m{X}}^{t})\|^2 \leq c_3^2(\mathcal{T}^{t})^2$ from \eqref{F-P2} further gives
\begin{equation}\label{Ttkap}
	\mathcal{T}^{t} \geq \frac{\kappa}{c_3}, ~\forall~ {t} \in  \mathcal{S}^o := \mathcal{S} \cap \{t: G({\m{X}}^t) - P^\infty \neq 0\}.
\end{equation}
{By \eqref{Tt} and the fact that $\Delta P^{t}$ converges to $0$ monotonically, we must have $\mathcal{T}^{t} \rightarrow 0$. Then, \eqref{Ttkap} implies $\mathcal{S}^o$ is an empty or a finite set. 
Thus, there exists some $t_2 \in \mathcal{S}$ such that 
\begin{equation*}
	G({\m{X}}^t) = P^\infty, ~\forall~ t \geq t_2.
\end{equation*}
Then, \eqref{DeltaP} directly yields $\Delta P^{t+1} \leq c_2(\mathcal{T}^t)^2$. 
Hence, similar to \textbf{Case I}, we can obtain
\begin{equation}\label{case3.2}
	\mathcal{T}^t \leq \sqrt{\Delta P^{0}} (\tau_2)^t,
\end{equation}
where $\tau_2=1/\sqrt{1+ \tilde{\omega}}$ and $\tilde{\omega} \in (0,1/c_2)$.}


Now, let $t_4>t_3$ such that $\tilde{\mathcal{S}}:=\{t_3,t_3+1,\ldots,t_4-1\} \subset \mathcal{S}$. 
Under Case II, it follows from the monotonic decrease of $\Delta P^t$ and \eqref{Tt2+} that
\begin{align*}
	2\sum_{t=t_3}^{t_4-1}\mathcal{T}^{t+1} \leq \sum_{t=t_3}^{t_4-1}\mathcal{T}^{t} + \frac{c_5^{\theta}}{1-\theta} \sum_{t=t_3}^{t_4-1} \left( (\Delta P^{t+1})^{1-\theta} - (\Delta P^{t+2})^{1-\theta} \right).
\end{align*}
By splitting the sum $\sum_{t=t_3}^{t_4-1}\mathcal{T}^{t} = \mathcal{T}^{t_3} + \sum_{t=t_3}^{t_4-2}\mathcal{T}^{t+1}$ and rearranging the terms, we deduce
\begin{equation*}
	\sum_{t=t_3}^{t_4-1}\mathcal{T}^{t+1} \leq \mathcal{T}^{t_3} - \mathcal{T}^{t_4} + \frac{c_5^{\theta}}{1-\theta}(\Delta P^{t_3+1})^{1-\theta}.
\end{equation*}
Then, for all the above three cases, using the above relation, and summing \eqref{case1} and \eqref{case3.2},
over $\tilde{\mathcal{S}}$, we can derive
\begin{equation}\label{sumt3}
	\sum_{t=t_3}^{t_4-1}\mathcal{T}^{t} \leq \max\left\{\frac{\sqrt{\Delta P^{0}}\tau_1^{t_3}}{1-\tau_1},2\mathcal{T}^{t_3}+\frac{c_5^{\theta} }{1-\theta}( \Delta P^{t_3})^{1-\theta}, \frac{\sqrt{\Delta P^{0}}\tau_2^{t_3}}{1-\tau_1}\right\},
\end{equation}
where $\tau_1 \in (0,1)$ and $\tau_2 \in (0,1)$,

Let $r$ be sufficiently small such that
$$
\mathcal{B}_r({\m{X}}^{\infty}):=\{{\m{X}}:\|{\m{X}}-{\m{X}}^{\infty}\|\leq r\} \subseteq \mathcal{N}_{\infty}.
$$
By {Theorem \ref{nonconvex_convergence}} and the fact that  $\m{1}_n(\bar{\m{x}}^{\infty})\tr=\m{X}^{\infty}$  is an accumulation point of $\m{X}^t$, there exists $ t_3' \in \mathcal{S}$ such that
\begin{align*}
	&\|{\m{X}}^{t_3'}-{\m{X}}^{\infty}\| < \frac{r}{2},~\text{and}\\ &\sqrt{c_1}\max\left\{\frac{\sqrt{\Delta P^{0}}\tau_1^{t_3'}}{1-\tau_1},2\mathcal{T}^{t_3'}+\frac{c_5^{\theta} }{1-\theta}( \Delta P^{t_3'})^{1-\theta},\frac{\sqrt{\Delta P^{0}}\tau_2^{t_3'}}{1-\tau_2}\right\}<\frac{r}{2}.
\end{align*}
Setting $t_3=t_3'$ in \eqref{sumt3} gives ${\m{X}}^{t_3} \in \mathcal{B}_r({\m{X}}^{\infty})$ and 
\begin{equation}\label{sumt}
		\sum_{t=t_3}^{t_4-1}\mathcal{T}^{t} \leq \frac{r}{2\sqrt{c_1}}.
\end{equation}
Next, we prove by contradiction that ${\m{X}}^{t} \in \mathcal{B}_r({\m{X}}^{\infty})\subseteq \mathcal{N}_{\infty}$ for all $t\geq t_3$. Suppose that there exists $t_4'>t_3$ (the smallest such index) such that $\|{\m{X}}^{t_4'}-{\m{X}}^{\infty}\|\geq r$. Then for $t\in \{t_3,\ldots,t_4'-1\} \subset \mathcal{S}$, setting $t_4=t_4'$  in \eqref{sumt3} yields
$$
\|{\m{X}}^{t_4}-{\m{X}}^{\infty}\| \leq \|{\m{X}}^{t_3}-{\m{X}}^{\infty}\|+\sum_{t=t_3}^{t_4-1}\|{\m{X}}^{t+1}-{\m{X}}^{t}\|<\frac{r}{2}+\sqrt{c_1}\sum_{t=t_3}^{t_4-1}\mathcal{T}^t<r,
$$
where the second and last inequalities follow respectively from \eqref{xt+1-xt} and \eqref{sumt}. This contradicts the assumption, so ${\m{X}}^{t} \in \mathcal{B}_r({\m{X}}^{\infty})$ for all $t \geq t_3$. Taking the limit as $t_4 \rightarrow \infty$ on both sides of \eqref{sumt} yields the summability of $\{\mathcal{T}^t\}$, and consequently that of $\{\|\m{X}^{t+1}-\m{X}^t\|\}$. Hence, $\{\m{X}^t\}$ is a Cauchy sequence and therefore is convergent.
\end{proof}
The following convergence rate results are immediate consequences of Theorem \ref{whole-convergence},
which can be directly derived using the standard techniques given in \cite{attouch2009convergence}.
\begin{corollary}
	Under the setting of Theorem~\ref{whole-convergence}, the following convergence rates hold:
	\begin{itemize}
		\item[1.] When $\theta \in (1/2,1)$, there exists some $d_1>0$ such that
		$$
		\|{\m{X}}^{t}-{\m{X}}^{\infty}\| \leq d_1 t^{-\frac{1-\theta}{2\theta-1}}, ~\forall~ t \geq 0.
		$$
		\item [2.] When $\theta \in (0,1/2]$, there exists some $d_2>0$ such that
		$$
		\|{\m{X}}^{t}-{\m{X}}^{\infty}\|\leq d_2(\tau_1)^{t} ,~\forall~ t \geq 0,
		$$
		where $\tau_1=1/\sqrt{{1+\omega}}$, $\omega \in (0,1/c_4)$, and
		$$
		c_4=\frac{3\kappa^{-\frac{1}{\theta}}\max\{1,L^2\}+n^2\max\left\{1,\frac{(1-\sigma_{A}^2)(1-\sigma_{C}^2)}{32(1+\sigma_{C}^2)L^2\sigma_{D}^2}\right\}}{n^2\gamma\min\{1,\alpha\}}.
		$$
		\item[3.] {When $\theta=0$, there exists some $d_3>0$ such that
		$$
		\|{\m{X}}^{t}-{\m{X}}^{\infty}\|\leq d_3(\tau_2)^{t} ,~\forall~ t \geq 0,
		$$
		where $\tau_2=1/\sqrt{1+\tilde{\omega}}$, $\tilde{\omega}\in (0,1/c_2)$, and
		$$
		c_2=\max\left\{1,\frac{(1-\sigma_{A}^2)(1-\sigma_{C}^2)}{32(1+\sigma_{C}^2)L^2\sigma_{D}^2}\right\}/(\gamma\min\{1,\alpha\}).
		$$}
		
	\end{itemize}
\end{corollary}
\begin{remark}
	{In the centralized setting, if the objective function satisfies the K{\L} property with exponent $\theta=0$, descent algorithms typically can achieve the convergence in finite number of iterations  \cite{attouch2010proximal}. However, in the decentralized setting, although the objective value $G(\m{X}^t)$ stabilizes at its limit $P^\infty$ after a finite number of steps (for all $t \geq t_2$ in the above proof), the consensus error and the gradient tracking error can only decay linearly due to the spectral properties of the mixing matrix, as established in Lemma~\ref{property W}. Consequently, the combined error metric $\mathcal{T}^t$ and the entire sequence $\{\m{X}^t\}$ can only have a linear convergence rate instead of having
convergence in finite iterations.}
\end{remark}

\subsection{Construction of Hessian Inverse Approximations}
We begin with briefly reviewing how the previously developed decentralized quasi-Newton methods construct bounded approximations of the Hessian inverse. 
The DQN method \cite{shorinwa2024distributed} simply assumes its Hessian inverse approximations generated by \eqref{DQN_BFGS} are uniformly positive definite. 
To guarantee global convergence in the nonconvex optimization setting, DGM-BB-C \cite{gao2022achieving} and DSG \cite{jakovetic2019exact} apply projection
to ensure positive definiteness of the Hessian inverse approximations (see \eqref{DSG-update}).
While the regularization and damping techniques in \cite{zhang2023variance}  were proposed in the convex optimization setting,
these techniques could be used for nonconvex optimization as well.
We now propose several new approaches to construct bounded Hessian inverse approximations  $\hat{\m{H}}_i^t$, specifically designed for nonconvex optimization frameworks. 
Below, we denote  $\m{s}_i^{t}=\m{x}_i^{t+1}-\m{x}_i^t$ and $\check{\m{y}}_i^{t}=\m{v}_i^{t+1}-\m{v}_i^t$,
and and let $0 < l \ll u < +\infty$ be two constants. 

\subsubsection{Memoryless quasi-Newton updates}

Because of low computational cost and easy implementation, memoryless quasi-Newton methods are widely used in traditional nonlinear optimization.
Here, we propose two practical memoryless quasi-Newton updating schemes for decentralized optimization. 
\paragraph{\it Memoryless SR1 update}
The memoryless SR1 method to directly approximate the Hessian inverse can be obtained by
\begin{equation}\label{MSR1}
	\hat{\m{H}}_i^{t+1}=\m{I}_p+\frac{(\m{s}_i^{t}-\check{\m{y}}_i^{t})(\m{s}_i^{t}-\check{\m{y}}_i^{t})\tr}{(\m{s}_i^{t}-\check{\m{y}}_i^{t})\tr \check{\m{y}}_i^{t}}.
\end{equation}
However, the main drawback of \eqref{MSR1} is that $\hat{\m{H}}_i^{t+1}$ cannot be maintained positive definite during
the iterations. Thus, with a necessary safeguard, \eqref{MSR1} can be modified to
\begin{align}\label{MSR1+}
		\m{H}_i^{t+1}=\left\{\begin{array}{cl}
	\m{I}_p+\frac{(\m{s}_i^{t}-\check{\m{y}}_i^{t})(\m{s}_i^{t}-\check{\m{y}}_i^{t})\tr}{(\m{s}_i^{t}-\check{\m{y}}_i^{t})\tr \check{\m{y}}_i^{t}}, &\text {if } [\lambda_{\min}(\hat{\m{H}}_i^{t+1}),\lambda_{\max}(\hat{\m{H}}_i^{t+1})]\subset  [l,u] \\ 
	&\quad \text{and } (\m{s}_i^{t}-\check{\m{y}}_i^{t})\tr \check{\m{y}}_i^{t} \neq 0; \\
		\m{I}_p,  &\text{otherwise. } 
	\end{array}\right.
\end{align}
As a rank-one perturbation on the identity matrix, the only eigenvalue of $\hat{\m{H}}_i^{t+1}$ other than one is $1+\frac{\|\m{s}_i^{t}-\check{\m{y}}_i^{t}\|^2}{(\m{s}_i^{t}-\check{\m{y}}_i^{t})\tr \check{\m{y}}_i^{t}}$. Thus, we can easily compute and check if the eigenvalues of $\hat{\m{H}}_i^{t+1}$ belong to the interval $[l,u]$.
From \eqref{MSR1+}, when $[\lambda_{\min}(\hat{\m{H}}_i^{t+1}),\lambda_{\max}(\hat{\m{H}}_i^{t+1})]\subset  [l,u]$ and $(\m{s}_i^{t}-\check{\m{y}}_i^{t})\tr \check{\m{y}}_i^{t} \neq 0$, 
the search direction  $\m{d}_i^{t+1}=-\m{H}_i^{t+1}\m{v}_i^{t+1}$ can be  written as
\begin{equation}\label{SR1-d}
	\m{d}_i^{t+1}=-\m{v}_i^{t+1}-\frac{(\m{s}_i^{t}-\check{\m{y}}_i^{t})\tr\m{v}_i^{t+1}}{(\m{s}_i^{t}-\check{\m{y}}_i^{t})\tr \check{\m{y}}_i^{t}}(\m{s}_i^{t}-\check{\m{y}}_i^{t}).
\end{equation}
Hence, it is very cheap for updating the search directions, since only three vector-to-vector inner products 
 $\|\m{s}_i^{t}-\check{\m{y}}_i^{t}\|^2$, $(\m{s}_i^{t}-\check{\m{y}}_i^{t})\tr\m{v}_i^{t+1}$ and $(\m{s}_i^{t}-\check{\m{y}}_i^{t})\tr \check{\m{y}}_i^{t}$
are needed in each iteration. 
\paragraph{\it Memoryless BFGS update} 
Defining the function $H_i^t: \Omega_i^t  \rightarrow \R^{p \times p}$ as
\begin{align}\label{H0}
	H_i^t(\m{y})=\tau_i^t\left(\m{I}_p-\frac{\m{s}_i^t(\m{y})\tr+\m{y}(\m{s}_i^t)\tr}{(\m{s}_i^t)\tr\m{y}}\right)
	+\left(1+\frac{\tau_i^t\|\m{y}\|^2}{(\m{s}_i^{t})\tr\m{y}}\right)\frac{\m{s}_i^t(\m{s}_i^t)\tr}{(\m{s}_i^{t})\tr\m{y}},
\end{align}
where $\Omega_i^t = \left\{\m{y} \in \R^{p}: (\m{s}_i^{t})\tr\m{y} >0 \right\}$, and letting
\begin{equation}\label{checks}
	\m{y}_i^t=\left\{\begin{array}{cl}
		\check{\m{y}}_i^t, & \text { if } [\lambda_{\min } (H_i^t(\check{\m{y}}_i^t)),\lambda_{\max } (H_i^t(\check{\m{y}}_i^t))] \subset  [l,u] \text{ and } (\m{s}_i^t)\tr\check{\m{y}}_i^t > 0 ; \\
		\hat{\m{y}}_i^t, & \text { otherwise, } 
	\end{array}\right.
\end{equation}
where $\hat{\m{y}}_i^{t}=\m{g}_i^{t+1}-\m{g}_i^t+h_i^t\m{s}_i^t$, $h_i^t=\varrho+\max\left\{-\frac{(\m{s}_i^t)\tr(\m{g}_i^{t+1}-\m{g}_i^t)}{\|\m{s}_i^t\|^2},0\right\}$
with $\varrho>0$ being a small constant, we propose a memoryless BFGS update as
\begin{align}\label{H1}
	\m{H}_i^{t+1}&=H_i^t(\m{y}_i^t)=\tau_i^t\left(\m{I}_p-\frac{\m{s}_i^t(\m{y}_i^t)\tr+\m{y}_i^t(\m{s}_i^t)\tr}{(\m{s}_i^t)\tr\m{y}_i^t}\right)
	+\left(1+\frac{\tau_i^t\|\m{y}_i^t\|^2}{(\m{s}_i^{t})\tr\m{y}_i^t}\right)\frac{\m{s}_i^t(\m{s}_i^t)\tr}{(\m{s}_i^{t})\tr\m{y}_i^t}\\
	&=\frac{(\m{s}_i^t)\tr \m{y}_i^t}{\|\m{y}_i^t\|^2} \m{I}_p - 
	\frac{\m{s}_i^t(\m{y}_i^t)\tr+\m{y}_i^t(\m{s}_i^t)\tr}{\|\m{y}_i^t\|^2} + 
	2 \frac{\m{s}_i^t(\m{s}_i^t)\tr}{(\m{s}_i^{t})\tr\m{y}_i^t},\notag
\end{align}
where $\tau_i^t = \frac{(\m{s}_i^t)\tr \m{y}_i^t}{\|\m{y}_i^t\|^2}$.

The adaptive selection mechanism in \eqref{checks} is to ensure the quasi-Newton matrix \eqref{H1} 
has strictly positive and bounded eigenvalues. 
Since $\m{v}_i^t$ is generated by the gradient tracking technique in Step 4 of Algorithm~\ref{alg:Framwork},
it captures some information of the average gradients on different nodes.
So, we prefer using $\m{v}_i^t$ to generate the quasi-Newton matrix by $H_i^t(\check{\m{y}}_i^t)$
whenever possible, where $\check{\m{y}}_i^{t}=\m{v}_i^{t+1}-\m{v}_i^t$.
However, $H_i^t(\check{\m{y}}_i^t)$ is not necessarily positive definite and bounded. 
So, in \eqref{checks} we simply check if the smallest and largest eigenvalues of the quasi-Newton
matrix $H_i^t(\check{\m{y}}_i^t)$ belong to the interval $[l,u]$, where 
$0 < l \ll u$ are two parameters. If not, we would use the alternative local gradient difference $\m{g}_i^{t+1}-\m{g}_i^t$ with correction $h_i^t\m{s}_i^t$ to 
update the quasi-Newton matrix by $H_i^t(\hat{\m{y}}_i^t)$, where
$\hat{\m{y}}_i^{t}=\m{g}_i^{t+1}-\m{g}_i^t+h_i^t\m{s}_i^t$.
By the BFGS updating formula \eqref{H1},  $H_i^t(\hat{\m{y}}_i^t)$ is guaranteed to be positive definite and uniformly bounded.
We now explain why the computational cost of the smallest and largest eigenvalues of
$\m{H}_i^{t+1}$ is in fact negligible. 
Note that $\m{H}_i^{t+1} = H_i^t({\m{y}}_i^t)$ is obtained 
from the scalar matrix $\tau_i^t \m{I}_p$ by a rank-two BFGS update.
Hence, if $\tau_i^t = \frac{(\m{s}_i^t)\tr \m{y}_i^t}{\|\m{y}_i^t\|^2}  >0 $,
$\m{H}_i^{t+1}$ has $p-2$ eigenvalues equal to $\tau_i^t$ and 
two eigenvalues $0< \lambda_i^{t+1} \le \Lambda_i^{t+1}$ satisfying
\begin{equation}\label{qua}
	\left\{\begin{array}{cl}
		\lambda_i^{t+1} \Lambda_i^{t+1} &= \frac{\|\m{s}_i^t\|^2}{\|\m{y}_i^t\|^2}, \\
		\lambda_i^{t+1}+ \Lambda_i^{t+1} &=\frac{2\|\m{s}_i^t\|^2}{(\m{s}_i^t)\tr \m{y}_i^t}.
	\end{array}\right.
\end{equation}
The system \eqref{qua} defines a quadratic equation, 
which has two roots
\begin{align}\label{lam_H}
	\left\{\begin{array}{cl}
		\lambda_i^{t+1} = \lambda_{\min } (H_i^t(\m{y}_i^t))
		=\frac{\|\m{s}_i^t\|^2}{(\m{s}_i^t)\tr {\m{y}}_i^t}\left(1- \sqrt{1- \frac{((\m{s}_i^t)\tr {\m{y}}_i^t)^2}{\|\m{s}_i^t\|^2\|{\m{y}}_i^t\|^2}} \right),\\
		~\\
		\Lambda_i^{t+1}	=	\lambda_{\max } (H_i^t(\m{y}_i^t))
		= \frac{\|\m{s}_i^t\|^2}{(\m{s}_i^t)\tr {\m{y}}_i^t}\left(1+ \sqrt{1- \frac{((\m{s}_i^t)\tr {\m{y}}_i^t)^2}{\|\m{s}_i^t\|^2\|{\m{y}}_i^t\|^2}} \right).
	\end{array}\right.
\end{align}
$\m{H}^t$ given as \eqref{H1} is easily shown to be bounded with $\tau_i^t=\frac{(\m{s}_i^t)\tr \m{y}_i^t}{\|\m{y}_i^t\|^2}$. From \eqref{checks}, we acquire $l\m{I}_p \preceq \m{H}_i^{t+1} \preceq u\m{I}_p$ when $\m{y}_i^t$ takes $\check{\m{y}}_i^t$. When $\m{y}_i^t$ takes $\hat{\m{y}}_i^t$, we have by \eqref{lam_H} that
\begin{equation*}
	\frac{(\m{s}_i^t)\tr \hat{\m{y}}_i^t}{2\|\hat{\m{y}}_i^t\|^2}\m{I}_p \preceq \m{H}_i^{t+1} \preceq \frac{2\|\m{s}_i^t\|^2}{(\m{s}_i^t)\tr \hat{\m{y}}_i^t}\m{I}_p.
\end{equation*}
Since $\|\hat{\m{y}}_i^{t}\|=\|\m{g}_i^{t+1}-\m{g}_i^t+h_i^t\m{s}_i^t\| \leq (2L+\varrho)\|\m{s}_i^t\|$ 
by the $L$-Lipschitz continuity of $\nabla f_i$, we have $(\m{s}_i^t)\tr \hat{\m{y}}_i^t \geq \varrho\|\m{s}_i^t\|^2 \geq \hat{\varrho} \|\hat{\m{y}}_i^t\|^2$ with $\hat{\varrho}=\frac{\varrho}{(2L+\varrho)^2}$, which implies
\begin{equation*}
	\min \left\{l,\frac{\hat{\varrho}}{2}\right\}\m{I} \preceq \m{H}^{t+1} \preceq \max \left\{u,\frac{2}{\varrho}\right\}\m{I}.
\end{equation*}
Hence, search direction $\m{d}_i^{t+1}=-\m{H}_i^{t+1}\m{v}_i^{t+1}$ given by the memoryless BFGS method can be written as
\begin{equation}\label{H1_direction}
	\m{d}_i^{t+1}=-\frac{(\m{s}_i^t)\tr \m{y}_i^t}{\|\m{y}_i^t\|^2} \m{v}_i^{t+1} +
	\frac{((\m{y}_i^t)\tr\m{v}_i^{t+1})\m{s}_i^t+((\m{s}_i^t)\tr\m{v}_i^{t+1})\m{y}_i^t}{\|\m{y}_i^t\|^2} - 
	2 \frac{(\m{s}_i^t)\tr\m{v}_i^{t+1}}{(\m{s}_i^{t})\tr\m{y}_i^t}\m{s}_i^t.
\end{equation}
Observe that it is also very cheap to update the search direction, since  only five vector-to-vector inner products 
 $\|\m{s}_i^t\|^2$, $\|\m{y}_i^t\|^2$, $(\m{s}_i^t)\tr {\m{y}}_i^t$, $(\m{s}_i^t)\tr {\m{v}}_i^{t+1}$, and $(\m{y}_i^t)\tr {\m{v}}_i^{t+1}$
are needed at each iteration.

\subsubsection{Corrected quasi-Newton updates}
Motivated by \cite{curtis2016self}, we let
\begin{equation}\label{brevey}
	\breve{\m{y}}^{t}_i=\eta_i^t\check{\m{y}}^{t}_i+(1-\eta_i^t)\m{s}^{t}_{i},
\end{equation}
where $\eta_i^t \in (0,1]$ is chosen such that 
\begin{equation}\label{twobound}
	 \frac{	(\m{s}^{t}_{i})\tr\breve{\m{y}}^{t}_i}{\|\m{s}^{t}_{i}\|^2} \geq \lambda
\end{equation}
for some $\lambda \in (0,1)$.
 The condition \eqref{twobound} is crucial for ensuring the positive definite Hessian inverse approximations in both centralized and decentralized settings.
We propose the following  adaptive strategy for choosing $\eta_i^t$ to ensure  \eqref{twobound} holds:
\begin{equation}\label{etait}
	\eta_i^t=\min\left\{\hat{\eta}_i^t,\frac{\hat{L}\|\m{s}^{t}_{i}\|}{\|\check{\m{y}}^{t}_i\|}\right\}
\end{equation}
where $\hat{L}$ is some positive constant  and $\hat{\eta}_i^t$ is defined as
\begin{equation}\label{hat_eta}
	\hat{\eta}_i^t=\left\{\begin{array}{cl}
		\frac{(1-\lambda)\|\m{s}^{t}_{i}\|^2}{\|\m{s}^{t}_{i}\|^2-(\m{s}^{t}_{i})\tr\check{\m{y}}^{t}_i}, & \text {if }(\m{s}^{t}_{i})\tr\check{\m{y}}^{t}_i \leq \lambda\|\m{s}^{t}_{i}\|^2; \\
		1, & \text {otherwise}.
	\end{array}\right.
\end{equation}
\begin{lemma}\label{lem 2.7}
	For $\eta_i^t$ defined in \eqref{etait} and $\breve{\m{y}}^{t}_i$ defined in \eqref{brevey}, we have $0<{\eta}_i^t
	\leq1$ and $(\m{s}^{t}_{i})\tr\breve{\m{y}}^{t}_i \geq \lambda\|\m{s}^{t}_{i}\|^2$.
\end{lemma}

\begin{proof}
	When $(\m{s}^{t}_{i})\tr\check{\m{y}}^{t}_i \leq \lambda\|\m{s}^{t}_{i}\|^2$, it can be readily verified that $0 < \hat{\eta}_i^t \leq 1$. Since $\eta_i^t = \min\left\{\hat{\eta}_i^t, \frac{\hat{L}\|\m{s}^{t}_{i}\|}{\|\check{\m{y}}^{t}_i\| }\right\}$, it follows immediately that $0 < \eta_i^t \leq 1$. 
	Taking the inner product of $\m{s}^{t}_{i}$ with both sides of \eqref{brevey}, we obtain
	\begin{align}\label{inner_g_eta}
		(\m{s}^{t}_{i})\tr\breve{\m{y}}^{t}_i
		= \eta_i^t \left( (\m{s}^{t}_{i})\tr\check{\m{y}}^{t}_i - \|\m{s}^{t}_{i}\|^2 \right) + \|\m{s}^{t}_{i}\|^2.
	\end{align}
 If $(\m{s}^{t}_{i})\tr\check{\m{y}}^{t}_i \leq \lambda\|\m{s}^{t}_{i}\|^2$, the term $((\m{s}^{t}_{i})\tr\check{\m{y}}^{t}_i - \|\m{s}^{t}_{i}\|^2)$ is strictly negative. Since $\eta_i^t \leq \hat{\eta}_i^t$, we have
	\begin{align*}
		(\m{s}^{t}_{i})\tr\breve{\m{y}}^{t}_i 
		&\geq \hat{\eta}_i^t \left( (\m{s}^{t}_{i})\tr\check{\m{y}}^{t}_i - \|\m{s}^{t}_{i}\|^2 \right) + \|\m{s}^{t}_{i}\|^2 \\
		&= \frac{(1-\lambda)\|\m{s}^{t}_{i}\|^2}{\|\m{s}^{t}_{i}\|^2-(\m{s}^{t}_{i})\tr\check{\m{y}}^{t}_i} \left( (\m{s}^{t}_{i})\tr\check{\m{y}}^{t}_i - \|\m{s}^{t}_{i}\|^2 \right) + \|\m{s}^{t}_{i}\|^2 \\
		&= -(1-\lambda)\|\m{s}^{t}_{i}\|^2 + \|\m{s}^{t}_{i}\|^2 = \lambda\|\m{s}^{t}_{i}\|^2.
	\end{align*}
If $(\m{s}^{t}_{i})\tr\check{\m{y}}^{t}_i > \lambda\|\m{s}^{t}_{i}\|^2$, we can rewrite \eqref{inner_g_eta} as a convex combination:
	\begin{align*}
		(\m{s}^{t}_{i})\tr\breve{\m{y}}^{t}_i = \eta_i^t (\m{s}^{t}_{i})\tr\check{\m{y}}^{t}_i + (1-\eta_i^t)\|\m{s}^{t}_{i}\|^2.
	\end{align*}
	Since $\eta_i^t \in (0, 1]$, and both $(\m{s}^{t}_{i})\tr\check{\m{y}}^{t}_i > \lambda\|\m{s}^{t}_{i}\|^2$ and $\|\m{s}^{t}_{i}\|^2 > \lambda\|\m{s}^{t}_{i}\|^2$ (because $\lambda < 1$), their convex combination must be strictly greater than $\lambda\|\m{s}^{t}_{i}\|^2$.
\end{proof}

The nonlinear conjugate gradient (NCG) methods proposed by Dai-Kou \cite{dai2013nonlinear} and Hager-Zhang \cite{hager2005new} are closely related to quasi-Newton methods. 
Extending the NCG search directions given in \cite{dai2013nonlinear,hager2005new} to the decentralized setting,  we can set 
\[
\tilde{\m{d}}_i^t=-\tilde{\m{H}}_i^t \m{v}_i^t,
\]
where 
\begin{align*}
	\tilde{\m{H}}_i^{t}=\m{I}-\frac{\m{s}^{t-1}_{i}(\m{z}^{t-1}_i)\tr}{(\m{s}^{t-1}_{i})\tr\breve{\m{y}}^{t-1}_i},~\m{z}^{t-1}_i=\breve{\m{y}}^{t-1}_i-\tau p^{t-1}_i\m{s}^{t-1}_i,~ p^{t-1}_i=\frac{\|\breve{\m{y}}^{t-1}_i\|^2}{(\m{s}^{t-1}_{i})\tr\breve{\m{y}}^{t-1}_i},
\end{align*}
and the parameter $\tau$ is suggested to lie in the interval $[1,2]$.
By symmetrizing $\tilde{\m{H}}_i^{t}$, we obtain a corrected quasi-Newton update for $\m{H}_i^t$:
\begin{equation}\label{DK_H1}
	{\m{H}}_i^{t}=\frac{\tilde{\m{H}}_i^{t}+(\tilde{\m{H}}_i^{t})\tr}{2}=\m{I}-\frac{1}{2}\frac{\m{s}^{t-1}_{i}(\m{z}^{t-1}_i)\tr+\m{z}^{t-1}_i(\m{s}^{t-1}_{i})\tr}{(\m{s}^{t-1}_{i})\tr\breve{\m{y}}^{t-1}_i},
\end{equation}
The search direction $\m{d}_i^t$ for this corrected quasi-Newton method is calculated by
\[
\m{d}_i^{t}=-{\m{H}}_i^{t}\m{v}_i^{t}=-\m{v}_i^{t}+\frac{(\m{z}^{t-1}_i)\tr\m{v}_i^{t}}{2(\m{s}^{t-1}_{i})\tr\breve{\m{y}}^{t-1}_i}\m{s}^{t-1}_{i}+\frac{(\m{s}^{t-1}_{i})\tr\m{v}_i^{t}}{2(\m{s}^{t-1}_{i})\tr\breve{\m{y}}^{t-1}_i}\m{z}^{t-1}_i.
\]
In the following lemma, we give uniform bounds on the eigenvalues of ${\m{H}}_i^{t}$ for any $i \in \{1,\ldots,n\}$ and $t \geq 0$. 

\begin{lemma}\label{lem 2.8}
	For ${\m{H}}_i^{t}$ defined by \eqref{DK_H1}, the eigenvalues of the Hessian inverse approximations are uniformly bounded as
	$$
	\frac{3}{4} \leq \lambda_{\min}({\m{H}}_i^{t}) \leq \lambda_{\max}({\m{H}}_i^{t}) \leq 1 + \frac{\sqrt{2(\hat{L}^2+1)}}{\lambda} + \frac{2\tau(\hat{L}^2+1)}{\lambda^2}.
	$$
\end{lemma}

\begin{proof}
	First, we establish the uniform upper bound. It is a standard property that $\lambda_{\max}({\m{H}}_i^{t}) \leq \|{\m{H}}_i^{t}\|_2$. By applying the triangle inequality and the submultiplicativity of the spectral norm, we obtain
	\begin{align}\label{H_norm_bound}
		\|{\m{H}}_i^{t}\|_2 &\leq \|\m{I}\|_2 + \frac{1}{2|(\m{s}^{t-1}_{i})\tr\breve{\m{y}}^{t-1}_i|} \left\| \m{s}^{t-1}_{i}(\m{z}^{t-1}_i)\tr + \m{z}^{t-1}_i(\m{s}^{t-1}_{i})\tr \right\|_2 \notag \\
		&\leq 1 + \frac{\|\m{s}^{t-1}_{i}\|\|\m{z}^{t-1}_i\|}{(\m{s}^{t-1}_{i})\tr\breve{\m{y}}^{t-1}_i},
	\end{align}
	where the second inequality uses the fact that $(\m{s}^{t-1}_{i})\tr\breve{\m{y}}^{t-1}_i > 0$ from Lemma \ref{lem 2.7}, and the relation $\|\m{u}\m{v}\tr + \m{v}\m{u}\tr\|_2 \leq 2\|\m{u}\|\|\m{v}\|$ for any $\m{u},\m{v} \in \R^p$.
	Applying the triangle inequality to the definition of $\m{z}^{t-1}_i$, we get
	$$
	\|\m{z}^{t-1}_i\| \leq \|\breve{\m{y}}^{t-1}_i\| + \tau \frac{\|\breve{\m{y}}^{t-1}_i\|^2}{(\m{s}^{t-1}_{i})\tr\breve{\m{y}}^{t-1}_i} \|\m{s}^{t-1}_{i}\|.
	$$
	Multiplying both sides by $\|\m{s}^{t-1}_{i}\| / ((\m{s}^{t-1}_{i})\tr\breve{\m{y}}^{t-1}_i)$ yields
	\begin{equation}\label{norm_H_bound2}
		\frac{\|\m{s}^{t-1}_{i}\|\|\m{z}^{t-1}_i\|}{(\m{s}^{t-1}_{i})\tr\breve{\m{y}}^{t-1}_i} \leq \frac{\|\m{s}^{t-1}_{i}\|\|\breve{\m{y}}^{t-1}_i\|}{(\m{s}^{t-1}_{i})\tr\breve{\m{y}}^{t-1}_i} + \tau \left( \frac{\|\m{s}^{t-1}_{i}\|\|\breve{\m{y}}^{t-1}_i\|}{(\m{s}^{t-1}_{i})\tr\breve{\m{y}}^{t-1}_i} \right)^2.
	\end{equation}
	From Lemma \ref{lem 2.7}, we know $(\m{s}^{t-1}_{i})\tr\breve{\m{y}}^{t-1}_i \geq \lambda\|\m{s}^{t-1}_{i}\|^2$. Thus, we can bound 
	\begin{equation}\label{norm_H_bound3}
		\frac{\|\m{s}^{t-1}_{i}\|\|\breve{\m{y}}^{t-1}_i\|}{(\m{s}^{t-1}_{i})\tr\breve{\m{y}}^{t-1}_i} \leq \frac{\|\m{s}^{t-1}_{i}\|\|\breve{\m{y}}^{t-1}_i\|}{\lambda\|\m{s}^{t-1}_{i}\|^2} = \frac{\|\breve{\m{y}}^{t-1}_i\|}{\lambda\|\m{s}^{t-1}_{i}\|}.
	\end{equation}
	Using the definition of $\breve{\m{y}}^{t-1}_i$ in \eqref{brevey}, Lemma \ref{important}, and the bound $\eta_i^{t-1} \leq  \frac{\hat{L}\|\m{s}^{t-1}_i\|}{\|\check{\m{y}}^{t-1}_i\|}$ from \eqref{etait}, we can bound 
	\begin{align}\label{norm_H_bound4}
		\|\breve{\m{y}}^{t-1}_i\|^2 &\leq 2(\eta_i^{t-1})^2\|\check{\m{y}}^{t-1}_i\|^2 + 2(1-\eta_i^{t-1})^2\|\m{s}^{t-1}_{i}\|^2 \notag\\
		&\leq 2\hat{L}^2\|\m{s}^{t-1}_{i}\|^2 + 2\|\m{s}^{t-1}_{i}\|^2 = 2(\hat{L}^2+1)\|\m{s}^{t-1}_{i}\|^2.
	\end{align}
	Substituting \eqref{norm_H_bound4}, \eqref{norm_H_bound3}, and \eqref{norm_H_bound2} into \eqref{H_norm_bound}, we obtain the upper bound.
	
	Next, to establish a uniform lower bound on the eigenvalues of ${\m{H}}_i^{t}$, we observe that
	\begin{equation}\label{minla}
		\lambda_{\min}({\m{H}}_i^{t})=\min\left\{1,1-\frac{1}{2}\left(\frac{(\m{s}^{t-1}_{i})\tr{\m{z}}^{t-1}_i}{(\m{s}^{t-1}_{i})\tr\breve{\m{y}}^{t-1}_i}+\frac{\|\m{s}^{t-1}_{i}\|\|{\m{z}}^{t-1}_i\|}{|(\m{s}^{t-1}_{i})\tr\breve{\m{y}}^{t-1}_i|}\right)\right\}.
	\end{equation}
	By defining
	$$
	q_i^{t-1}=\frac{\|\breve{\m{y}}^{t-1}_i\|^2\|{\m{s}}^{t-1}_i\|^2}{((\m{s}^{t-1}_{i})\tr\breve{\m{y}}^{t-1}_i)^2},
	$$
	we can rewrite \eqref{minla} as
	\begin{equation}\label{min_lambda}	
		\lambda_{\min}({\m{H}}_i^{t})=\min\left\{1,\frac{1}{2}\left(1+\tau q_i^{t-1}-\sqrt{\tau^2 (q_i^{t-1})^2-2 \tau q_i^{t-1} + q_i^{t-1}}\right)\right\}.
	\end{equation}
	The second term on the right-hand side of \eqref{min_lambda} is monotonically decreasing when $q_i^{t-1} \geq 1$. As $q_i^{t-1}$ tends to $+\infty$, its limit approaches $1 - \frac{1}{4\tau} \geq \frac{3}{4}$ due to the choice of $\tau \in [1,2]$.
	Hence, we always have $\lambda_{\min}({\m{H}}_i^{t}) \geq 3/4$.
	Combining this with the uniform upper bound established earlier, the proof is complete.
\end{proof}

\section{NUMERICAL EXPERIMENTS}
In this section, we would examine the performance of our developed algorithms in the following outline:
\begin{itemize}
\item[a.] Compare UDOAs using different approximate $\m{H}^t$s to Hessian inverse;
\item[b.] Compare UDOAs with some well-developed nonconvex optimization algorithms, including gradient-based algorithms \cite{nedic2017achieving,NEURIPS2023_98f8c89a,takezawa2022momentum,huang2024accelerated} and quasi-Newton algorithms \cite{shorinwa2024distributed,zhang2023variance}.
\end{itemize}
The optimization problem \eqref{obj_fun1} is smooth but nonconvex, over a connected undirected network with edge density $d \in (0,1]$.
For the generated network, we choose the Metropolis constant edge weight matrix \cite{xiao2007distributed} as the mixing matrix with
\begin{equation*}
	\tilde{W}_{i j}=\left\{\begin{array}{cl}
		\frac{1}{\max \{\operatorname{deg}(i), \operatorname{deg}(j)\}+1}, & \text { if }(i, j) \in \mathcal{E}; \\
		0, & \text { if }(i, j) \notin \mathcal{E} \text { and } i \neq j; \\
		1-\sum_{k \in \mathcal{N}_i/ \{i\}} \tilde{W}_{i k}, & \text { if } i=j,
	\end{array}\right.
\end{equation*}
where $(i, j) \in \mathcal{E}$ indicates there is an edge between 
node $i$ and node $j$, and $\operatorname{deg}(i)$ is the degree of node $i$. 
In our experiments, we introduce the communication volume required by an method, which can be calculated as 
\begin{align*}\notag
	\text{Communication volume} = & 
	\text{~iteration number}  \times \text{~the number of edges $\frac{dn(n-1)}{2}$} \\
	&\times  \text{~number of communication rounds per iteration} \\
	& \times   \text{~dimension of transmitted vectors on each edge}.
\end{align*}
In all experiments, we set the number of nodes $n=10$ and the edge density $d=0.56$ for the network. For all comparison algorithms, we initialize $\m{x}^0=\m{0}$. 
All experiments are coded in MATLAB R2017b and run on a laptop with Intel Core i5-9300H CPU, 16GB RAM, and Windows 10 operating system.

In particular, we consider the nonconvex decentralized binary classification problem. Using a logistic regression formulation with a nonconvex regularization, the optimization problem is given by
\begin{equation}\label{nonconvex_logistic_problem}
	\mathop {\min }\limits_{\m{z} \in {\mathbb{R}^p}} \sum_{i=1}^n  \sum_{j=1}^{m_i} \log \left(1+\exp (-b_{ij} \m{a}_{ij}\tr \m{z} ) \right)+\hat{\lambda} \sum_{k=1}^p \frac{\m{z}_{[k]}^2}{1+\m{z}_{[k]}^2},
\end{equation}
where $\m{a}_{ij} \in \mathbb{R}^p$ is the feature vector, $ b_{ij} \in \{-1,+1\}$ is the label, $\m{z}_{[k]}$ denotes the $k$-th component of $\m{z}$, and the regularization parameter is set to $\hat{\lambda} =1$.
The logistic loss function is the semi-algebraic function \cite{li2018calculus}. The regularization term is a rational function and hence, is also semi-algebraic. The sum of semi-algebraic functions remains semi-algebraic, which satisfies the K{\L} property. Thus, this objective function in \eqref{nonconvex_logistic_problem} satisfies the K{\L} property in ${\mathbb{R}^p}$.
From the first-order stationary condition \eqref{stationarity}, the success of each algorithm is measured by 
\begin{equation*}
\mbox{Optimality error} :=	\left\| \frac{1}{n}\sum_{i=1}^n\nabla f_i(\m{x}^t_{i}) \right\|+\|\m{x}^t-\m{M}\m{x}^t\|.
\end{equation*}
We utilize six datasets from the LIBSVM library, detailed in Table~\ref{table2}, and categorize these datasets into two groups:
\begin{itemize}
	\item \textbf{Large Sample and Low Dimension:} \textbf{mushrooms}, \textbf{ijcnn1}, \textbf{w8a}, and \textbf{a9a}. These datasets feature a large number of samples, mainly used to test convergence rate and communication efficiency.
	\item \textbf{Small Sample and High Dimension:} \textbf{colon-cancer} and \textbf{duke breast-cancer}. These datasets have a high feature dimension ($p \gg n$), which serves to evaluate the computational efficiency (CPU time) of the algorithms when Hessian approximations become computationally intensive.
\end{itemize}
\begin{table}[H]
	\caption{Datasets}\label{table2}
	\centering
	\begin{tabular}{ccc}
		\hline
		\hline
		{Dataset}&\# of samples ($\sum_{i=1}^n m_i$)  &\# of features ($p$)  \\
		\hline
		\textbf{mushrooms}&8124 &112\\
		\textbf{ijcnn1}&49990&22\\
		\textbf{w8a}&49749&300\\
		\textbf{a9a}&32561&123\\
		\textbf{colon-cancer}&62&2000\\
		\textbf{duke breast-cancer}&44&7129\\
		\hline
		\hline
	\end{tabular}
\end{table}
We first investigate the performance of UDOA with different approximations to Hessian inverse, namely different $\m{H}_i^t$s. Based on the proposed techniques in Section 2.3, four variants of UDOA 
are obtained and represented as UDOA($j$), $j=1,2,3,4$. 
\begin{itemize}
	\item UDOA(1) implements the memoryless SR1 update. Its detailed procedure is given in Algorithm~\ref{UDOA1} (see Appendix B).
	\item UDOA(2) corresponds to the memoryless BFGS method, outlined in Algorithm \ref{UDOA2} (see Appendix B).
	\item UDOA(3) and UDOA(4) represent Dai-Kou type and Hager-Zhang type corrected quasi-Newton methods, respectively. Both procedures are summarized in Algorithm~\ref{UDOA3} (see Appendix B).
\end{itemize}
For datasets \textbf{mushrooms}(\textbf{ijcnn1};\textbf{w8a};\textbf{a9a};\textbf{colon-cancer};\textbf{duke breast-cancer}), the following parameters are set, respectively, according to their better performance:
\begin{itemize}
	\item $\alpha=0.12(0.24;0.21;0.16;0.038;0.022)$, $l=10^{-6}$, and $u=10^6$ in UDOA(1);
	\item $\alpha=0.22(0.32;0.26;0.34;0.14;0.2)$, $\varrho=0.05(0.05;0.01;0.01;0.3;0.5)$, $l=10^{-6}$, and $u=10^6$ in UDOA(2);
	\item $\alpha=0.09(0.14;0.11;0.11;0.013;0.0072)$, $\lambda=0.7(0.8;0.8;0.8;0.7;0.7)$ and $\hat{L}=1(1;3;2;4;4)$ in UDOA(3);
	\item $\alpha=0.05(0.09;0.06;0.07;0.009;0.0046)$, $\lambda=0.7(0.8;0.8;1.1;0.8;0.8)$ and $\hat{L}=2(1;1;2;3;3)$ in UDOA(4).
\end{itemize} 
Fig.~\ref{nonconvex_UDOA} shows the curves of optimality error versus communication volume on the different datasets. Here, it is important to highlight that the communication and computational overheads per iteration remain largely consistent among these UDOA variants. The performance of these algorithms is significantly different. 
UDOA(2) converges fastest on the datasets with a large number of samples, while UDOA(3) wins on 
datasets with high feature dimensionality. 
The rankings of different algorithms are given in Table \ref{table3}, which shows
UDOA(2) and UDOA(1) perform on average the best and worst, respectively.
Hence, we use UDOA(1) and UDOA(2) in the following numerical comparisons with other algorithms.
\begin{figure*}[!t]
	\centering
	\subfloat[\textbf{mushrooms}]{\includegraphics[width=2.5in]{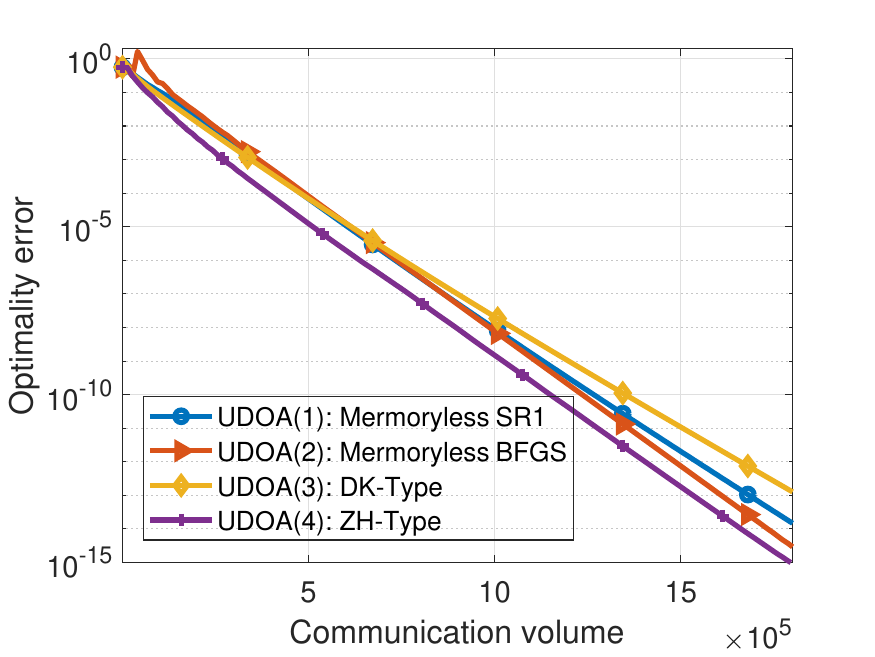}%
		\label{fig_a}}
	\subfloat[\textbf{ijcnn1}]{\includegraphics[width=2.5in]{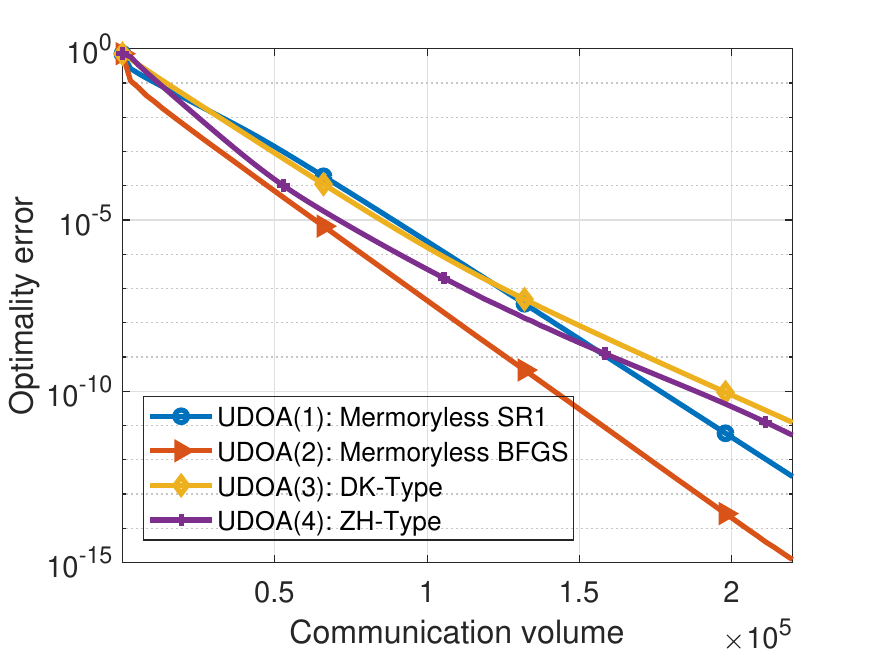}%
		\label{fig_b}}
	\hfil
	\subfloat[\textbf{w8a}]{\includegraphics[width=2.5in]{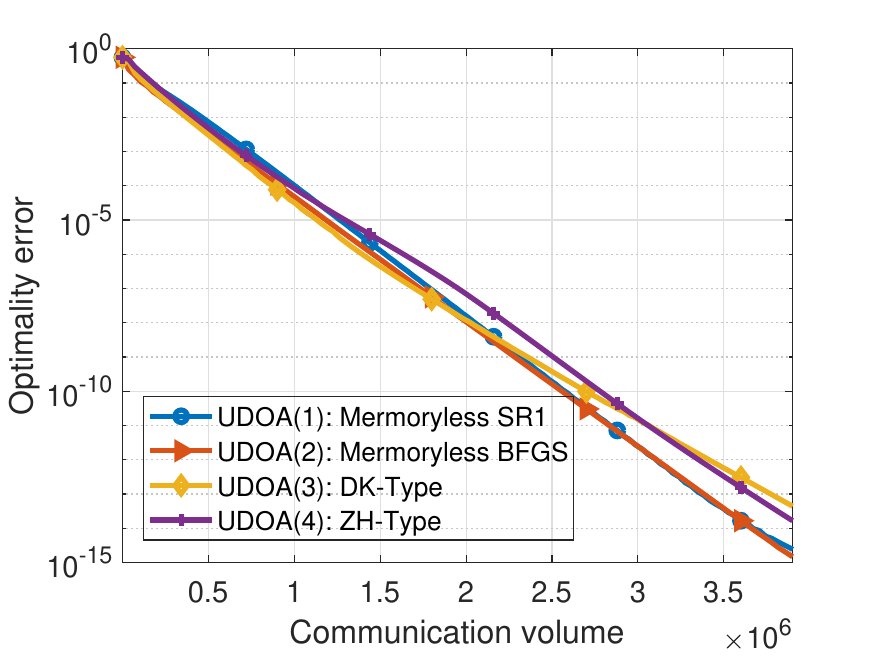}%
		\label{fig_c}}
	\subfloat[\textbf{a9a}]{\includegraphics[width=2.5in]{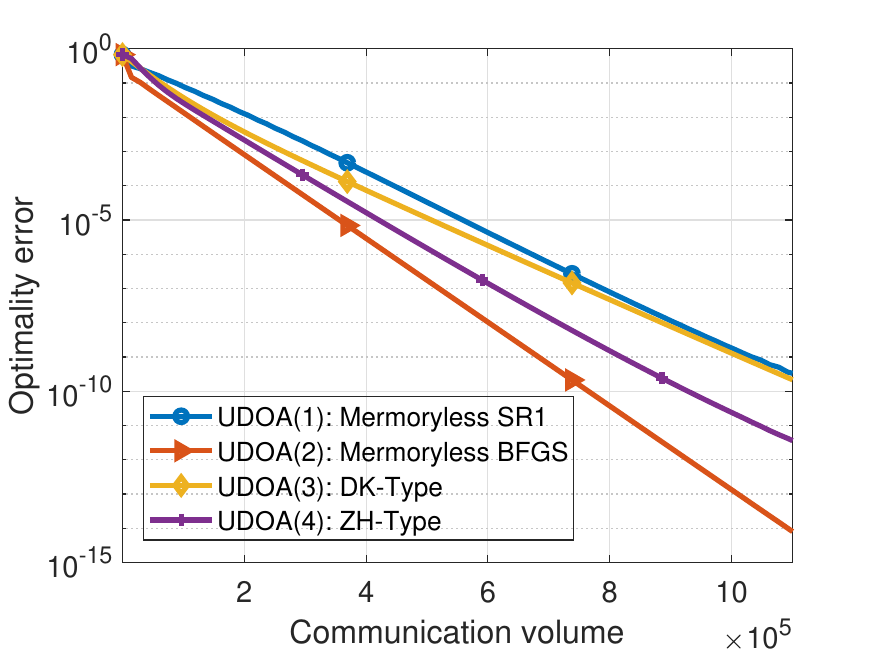}%
		\label{fig_d}}
	\hfil
	\subfloat[\textbf{colon-cancer}]{\includegraphics[width=2.5in]{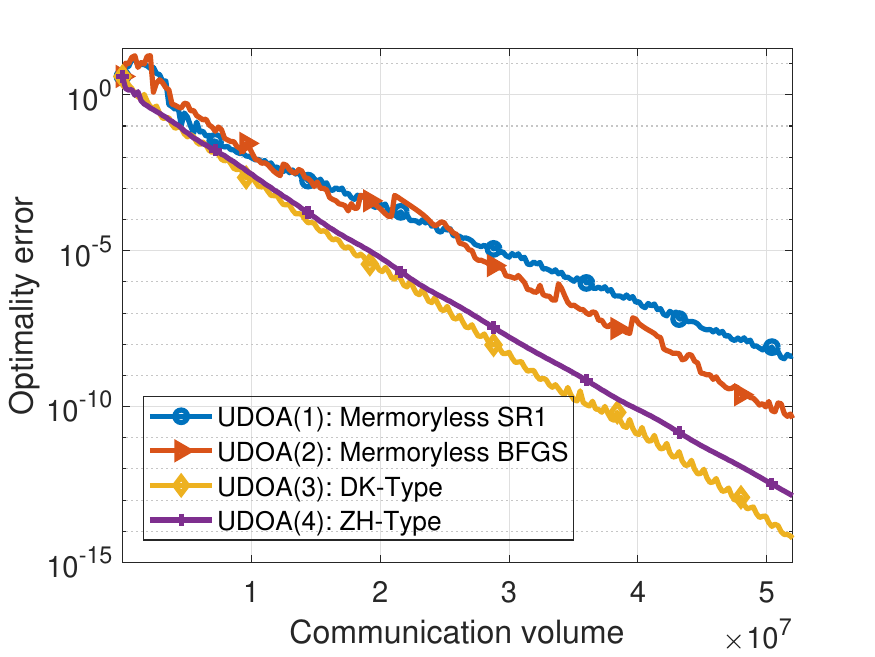}%
		\label{fig_e}}
	\subfloat[\textbf{duke breast-cancer}]{\includegraphics[width=2.5in]{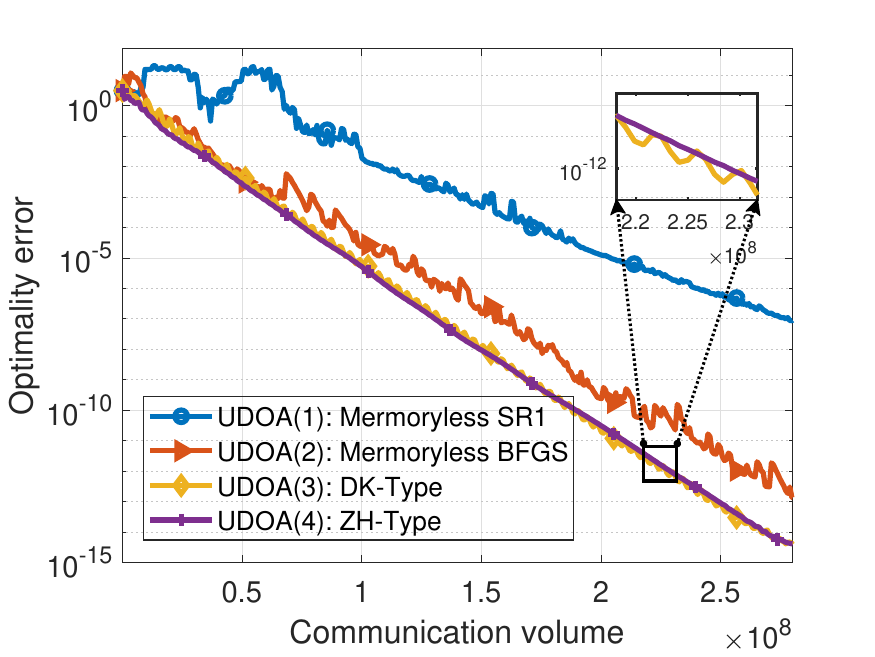}%
		\label{fig_f}}
	\caption{Optimality error of UDOAs for minimizing the nonconvex logistic regression problem \eqref{nonconvex_logistic_problem} on different datasets against communication volume.}
	\label{nonconvex_UDOA}
\end{figure*}
\begin{table}[H]
	\caption{Rankings of UDOAs on different datasets}\label{table3}
	\centering
	\begin{tabular}{cccccc}
		\hline
		\hline
		&UDOA(1)&UDOA(2)&UDOA(3)&UDOA(4)\\
		\hline
		\textbf{mushrooms} & 3& 2&4&1\\
		\textbf{ijcnn1}&2&1&4&3\\
		\textbf{w8a}&2&1&4&3\\
		\textbf{a9a}&4&1&3&2\\
		\textbf{colon-cancer}&4&3&1&2\\
		\textbf{duke breast-cancer}&4&3&1&2\\
		average&3.17&1.83&2.83&2.17\\
		\hline
		\hline
	\end{tabular}
\end{table}
\subsection{On Large Sample Datasets}
We now compare  UDOA(1) and UDOA(2) with the following well-established 
gradient-tracking-type algorithms:
Gradient Tracking (GT) \cite{nedic2017achieving}, Global Update Tracking ({GUT}) \cite{NEURIPS2023_98f8c89a}, Momentum Tracking ({MT}) \cite{takezawa2022momentum}, Distributed Stochastic Momentum Tracking ({DSMT}) \cite{huang2024accelerated}, on the four datasets: \textbf{mushrooms}(\textbf{ijcnn1};\textbf{w8a};\textbf{a9a}). 
 Although {GUT}, {MT}, and {DSMT} are stochastic methods, since we focuses on comparing 
deterministic decentralized methods,  full gradients are used for these methods.
The following parameters are set according to each algorithm's best performance.
Following the notations in their source papers, for each of the above four datasets,
we set $\eta=0.06(0.09;0.09;0.08)$ in GT; set $\eta_t=(0.01;0.01;0.01;0.02)\times n^{1/2}/t^{1/3}$ and $\mu=0.3$ in GUT; set $\eta=0.05$ and $\beta=0.31(0.41;0.33;0.35)$ in MT: set $\eta_{w}=1/(1+\sqrt{1-(1-\sigma)^2})$, $\beta=1-(1-\sqrt{\eta_{w}})/n^{1/3}$, and $\alpha=0.04(0.08;0.04;0.08)$ in DSMT, where $\sigma$ is defined in \textbf{Lemma \ref{property W}}. 
All algorithms except GUT need two rounds of communication per iteration. GUT needs only one round communication per iteration but uses a decreasing stepsize, which yields slow convergence as shown in
Fig.~\ref{nonconvex}. Hence, only optimality error against communication volume is shown in 
Fig.~\ref{nonconvex}, where we can see UDOA(1) and UDOA(2) are significantly better than GT and momentum-based methods, including MT and DSMT, for this set of nonconvex classification problems. 

\begin{figure*}[!t]
	\centering
	\subfloat[\textbf{mushrooms}]{\includegraphics[width=2.5in]{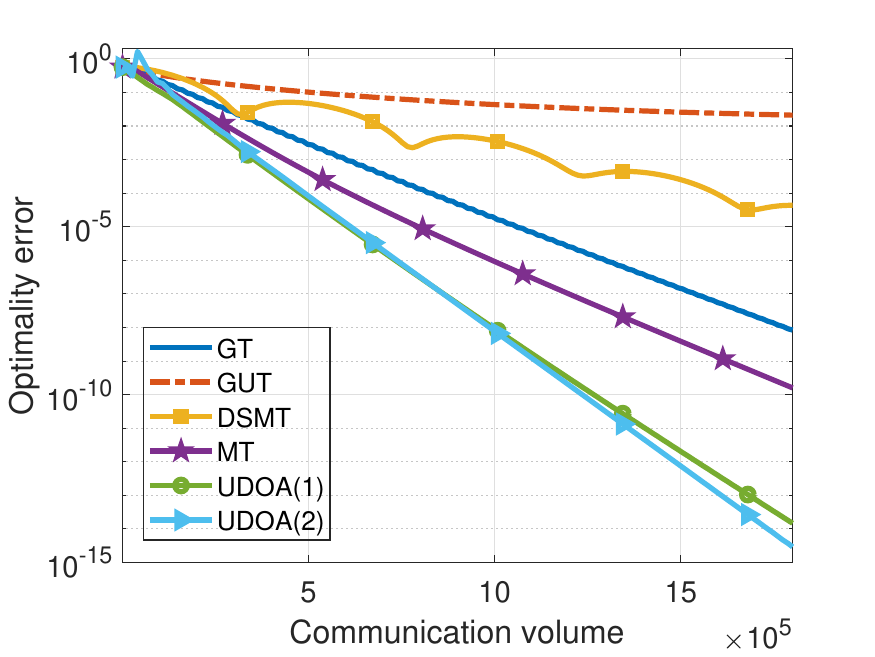}%
		\label{fig_n_a}}
	\subfloat[\textbf{ijcnn1}]{\includegraphics[width=2.5in]{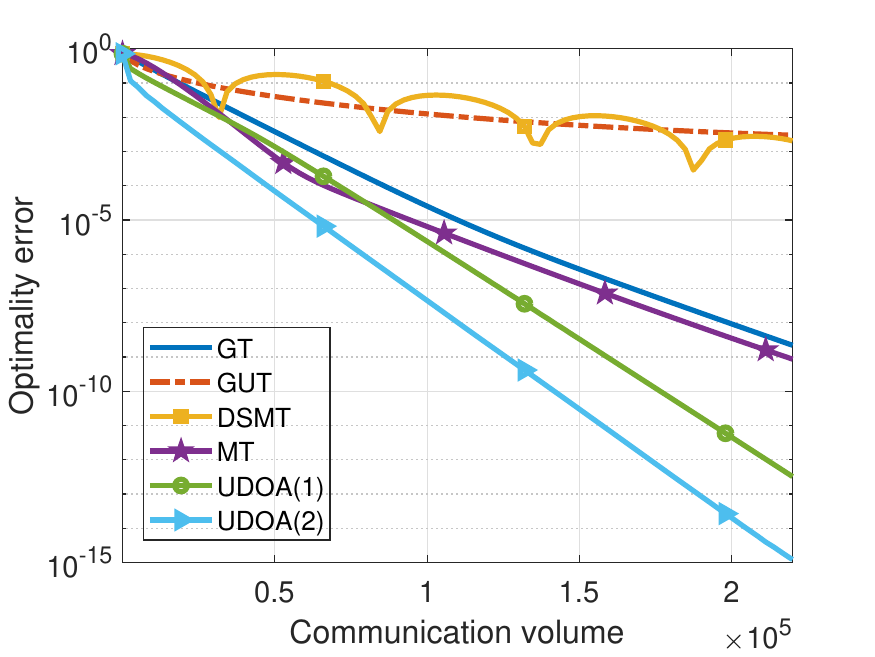}%
		\label{fig_n_b}}
	\hfil
	\subfloat[\textbf{w8a}]{\includegraphics[width=2.5in]{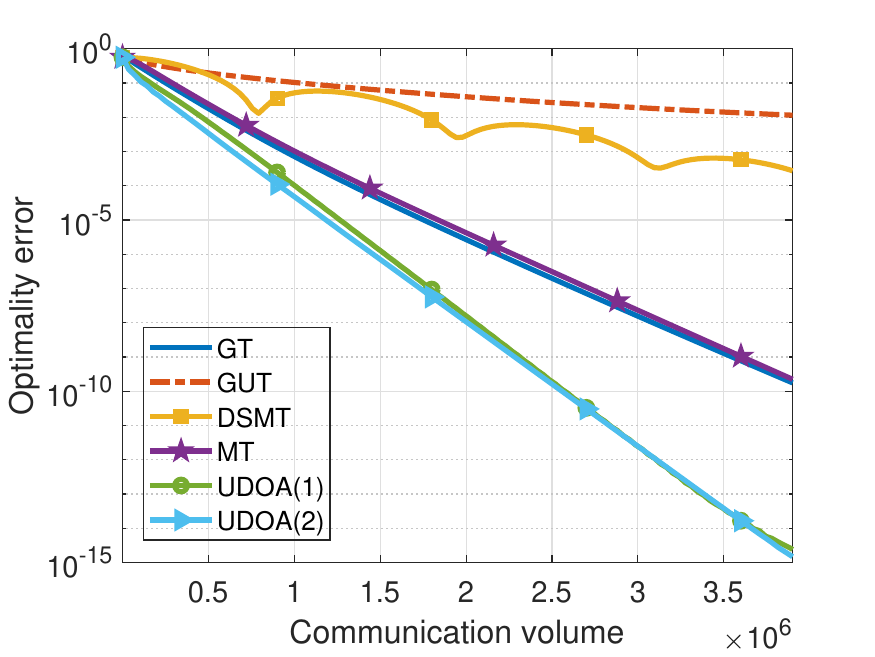}%
		\label{fig_n_c}}
	\subfloat[\textbf{a9a}]{\includegraphics[width=2.5in]{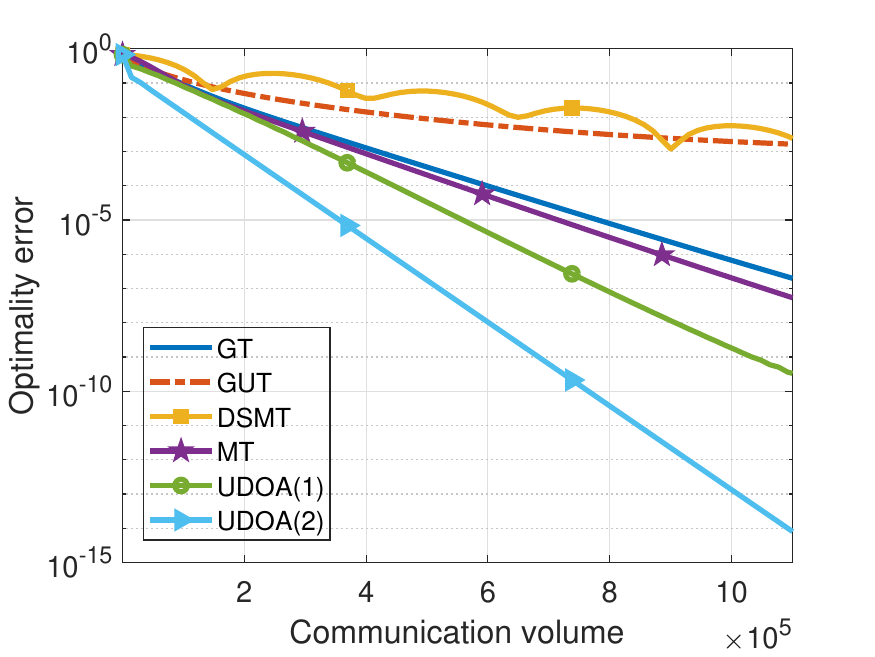}%
		\label{fig_n_d}}
	\caption{Optimality error gradient-tracking-type algorithms for minimizing the nonconvex logistic regression problem \eqref{nonconvex_logistic_problem} on different datasets against communication volume.}
	\label{nonconvex}
\end{figure*}

We next compare  UDOA(1) and UDOA(2) with three well-developed decentralized quasi-Newton algorithms: 
DQN \cite{shorinwa2024distributed}, DR-LM-DFP \cite{zhang2023variance}, and D-LM-BFGS \cite{zhang2023variance}, 
on the four datasets with large sample data: \textbf{mushrooms} (\textbf{ijcnn1}; \textbf{w8a}; \textbf{a9a}).
Again, for each data set, using  the notations in their source papers,
the following parameters are set according to each algorithm's best performance.
In particular, we set $\alpha_i^t=a_* \zeta_i^t$ and $\gamma=3(3;2;3)$ with $a_*=0.28(0.29;0.29;0.3)$ for DQN, where $\zeta_i^t \sim \mbox{U}(0.5,1.5)$ is the random variable generated over node $i$ at iteration $t$ and satisfies the uniform distribution over interval $(0.5, 1.5)$. 
For D-LM-BFGS, we set $\alpha=0.42 (0.34; 0.28; 0.36)$, $\epsilon=10^{-3}$, $\beta=10^{-3}$, $\mathcal{B}=10^4$, $\tilde{L}=5$, $M=4 (4; 5; 8)$.
For DR-LM-DFP, we set $\alpha=0.11 (0.18; 0.16; 0.16)$, $\rho=0.05 (0.04; 0.08; 0.05)$, $\epsilon=10^{-3}$, $\beta=1(10^{-3};10^{-3};10^{-3})$, $\mathcal{B}=10^4$, $\tilde{L}=5$, $M=5 (5; 4; 5)$.
The optimality errors against communication volume  are shown in Fig. \ref{quasi_iter}, where we can see that
in terms of iteration number, DQN performs best for the \textbf{mushrooms} and \textbf{w8a} datasets
 while UDOA(2) performs best for the \textbf{ijcnn1} and \textbf{a9a} datasets.
However, as shown in Fig \ref{quasi_com}, both UDOA(1) and UDOA(2) are significantly efficient in terms of communication volume among the comparison algorithms.
Although having competitive performance in terms of iteration number, DQN demonstrates significant
limitations in terms of communication cost, since three rounds of communication are needed in 
each DQN's iteration.
\subsection{On High Dimension Datasets}
In this subsection, we would like to compare the overall CPU time efficiency of  UDOA(1) and UDOA(2) 
with the state-of-the-art gradient-based and quasi-Newton algorithms. GT, MT, DQN, and D-LM-BFGS are adopted as comparison algorithms, serving as representatives of gradient-based, momentum-based, standard BFGS, and limited-memory BFGS methods, respectively.
Compared with DR-LM-DFP, the quasi-Newton matrix in D-LM-BFGS can be updated by a more efficient
a two-loop recursion, which is crucial for solving problems \eqref{nonconvex_logistic_problem} on datasets 
with large features.

Note that gradient computations dominate the computational cost of the algorithm for solving
\eqref{nonconvex_logistic_problem}
on large sample datasets (\textbf{mushrooms}, \textbf{ijcnn1}, \textbf{w8a}, and \textbf{a9a}). 
Since all comparison algorithms perform a same number of gradient evaluations per iteration, their per-iteration computational costs are comparable. Consequently, on these datasets, the total CPU time is largely proportional to the number of iterations.
So, we now focus this comparison on the \textbf{colon-cancer} and \textbf{duke breast-cancer} datasets, which have 
large features in each sample. In this regime, the efficiency of the Hessian approximation strategy plays a decisive role in the total run time.

For \textbf{colon-cancer} (\textbf{duke breast-cancer}) datasets, 
the following parameters are set according to each algorithm's best performance:
we set $\eta=0.0027(0.0013)$ in GT; set $\eta=0.0035(0.0017)$ and $\beta=0.8(0.89)$ in MT; set $\alpha_i^t=a_* \zeta_i^t$ and $\gamma=2$ with $a_*=0.12(0.1)$ in DQN, where $\zeta_i^t \sim \mbox{U}(0.5,1.5)$; set $\alpha=0.44 (0.32)$, $\epsilon=10^{-3}$, $\beta=10^{-3}$, $\mathcal{B}=10^4$, $\tilde{L}=8(10)$, $M=13(10)$ in D-LM-BFGS.
Again, Figs. \ref{colon} and \ref{breast} show that  UDOA(1) and UDOA(2) are the most efficient in CPU time.
Although DQN performs well in terms of iteration number, it is much slower in CPU time,
since a significant amount of  matrix-vector products are needed in each iteration for communication.

\begin{figure*}[!t]
	\centering
	\subfloat[\textbf{mushrooms}]{\includegraphics[width=2.5in]{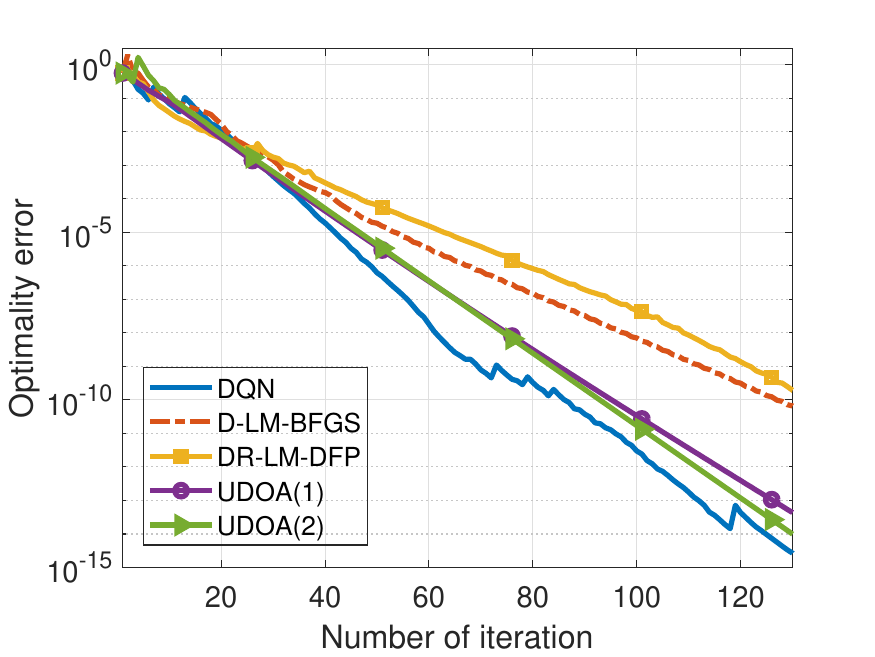}%
		\label{fig_a_iter}}
	\subfloat[\textbf{ijcnn1}]{\includegraphics[width=2.5in]{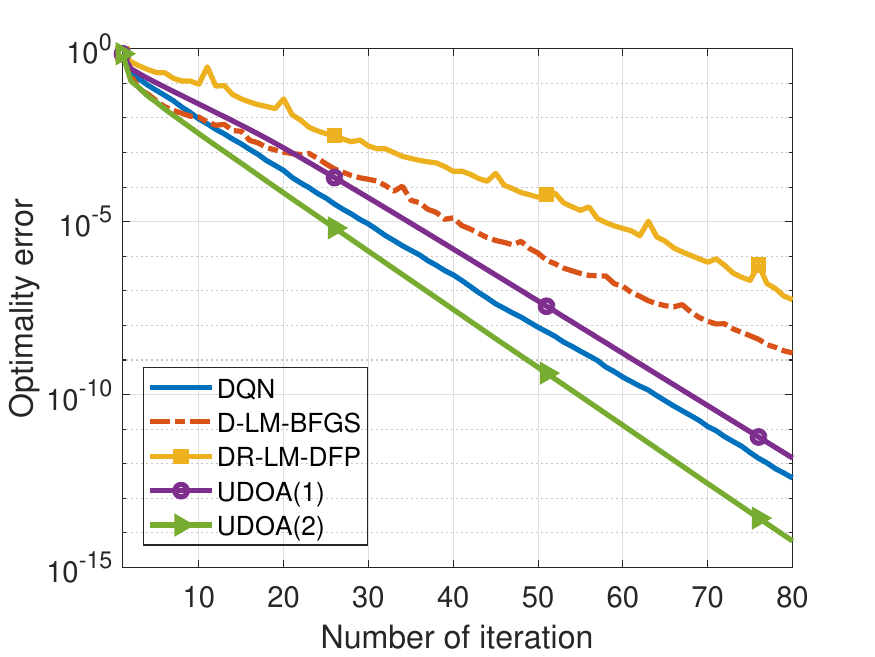}%
		\label{fig_b_iter}}
	\hfil
	\subfloat[\textbf{w8a}]{\includegraphics[width=2.5in]{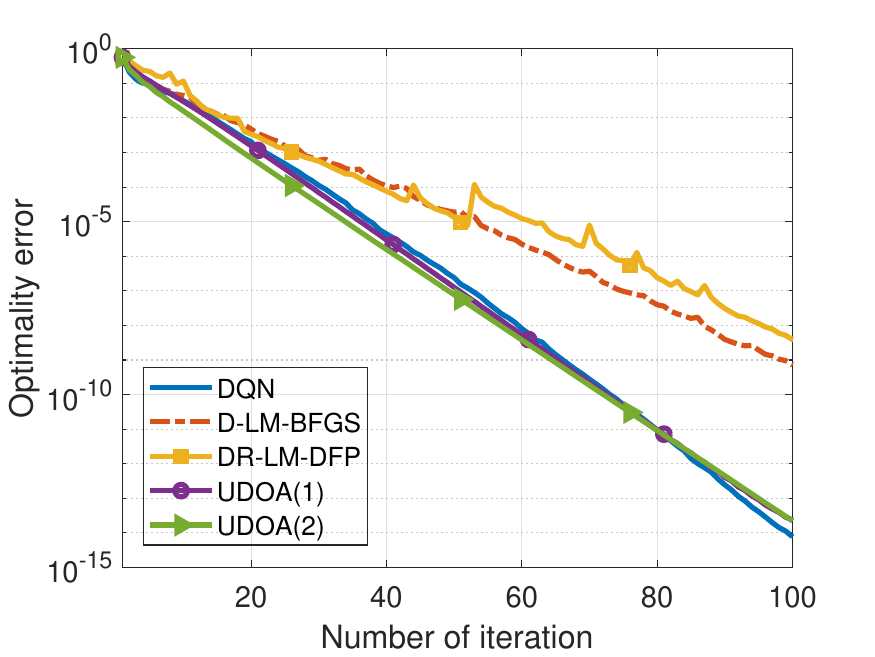}%
		\label{fig_c_iter}}
	\subfloat[\textbf{a9a}]{\includegraphics[width=2.5in]{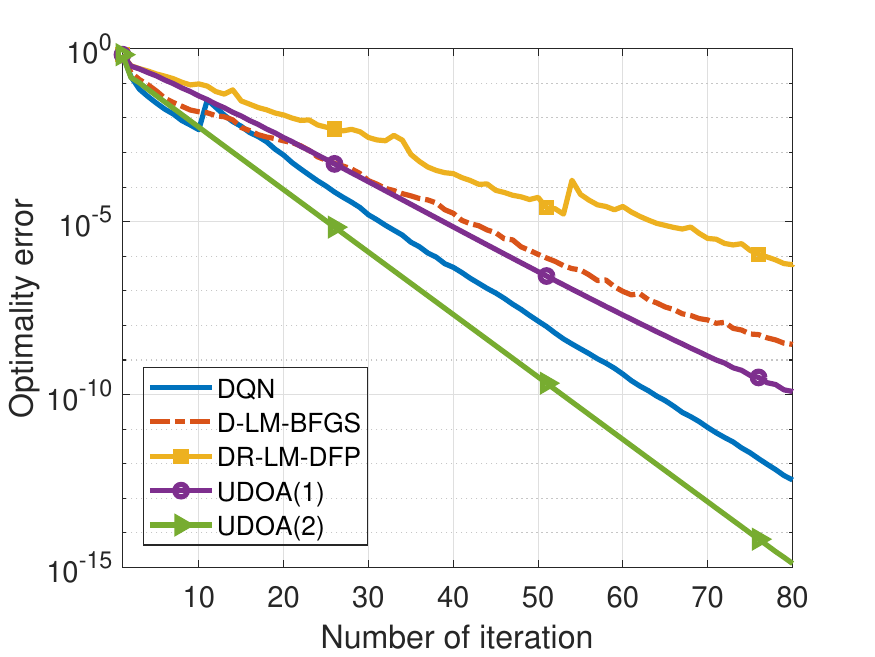}%
		\label{fig_d_iter}}
	\caption{Optimality error of quasi-Newton decentralized algorithms for minimizing the nonconvex logistic regression problem \eqref{nonconvex_logistic_problem} on different datasets against number of iteration.}
	\label{quasi_iter}
\end{figure*}

\begin{figure*}[!t]
	\centering
	\subfloat[\textbf{mushrooms}]{\includegraphics[width=2.5in]{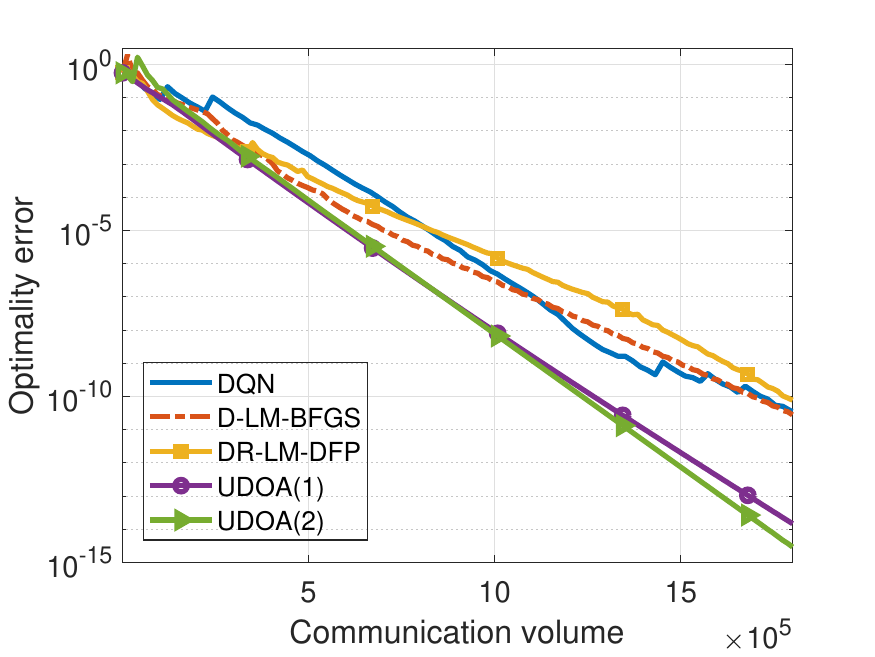}%
		\label{fig_a_com}}
	\subfloat[\textbf{ijcnn1}]{\includegraphics[width=2.5in]{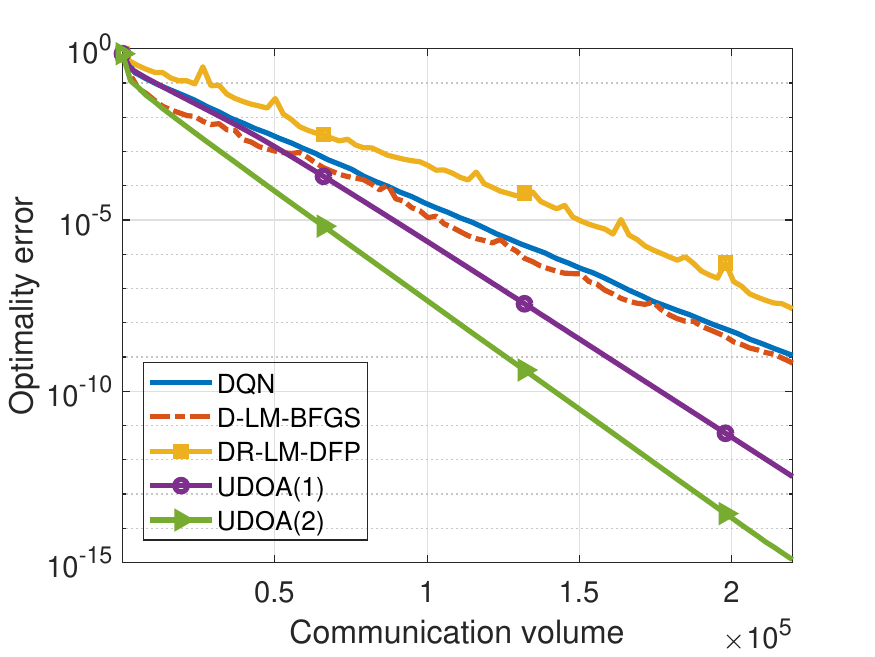}%
		\label{fig_b_com}}
	\hfil
	\subfloat[\textbf{w8a}]{\includegraphics[width=2.5in]{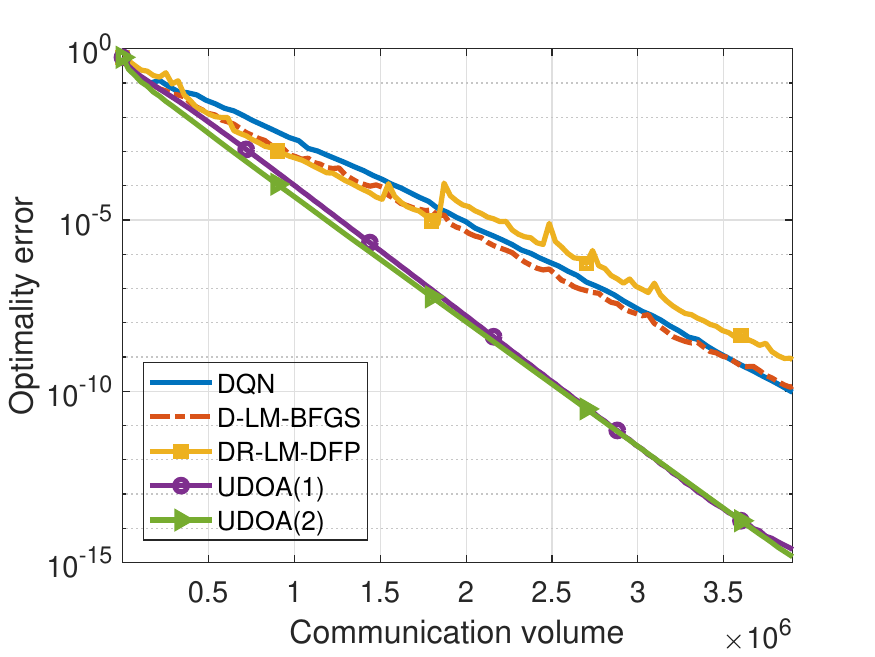}%
		\label{fig_c_com}}
	\subfloat[\textbf{a9a}]{\includegraphics[width=2.5in]{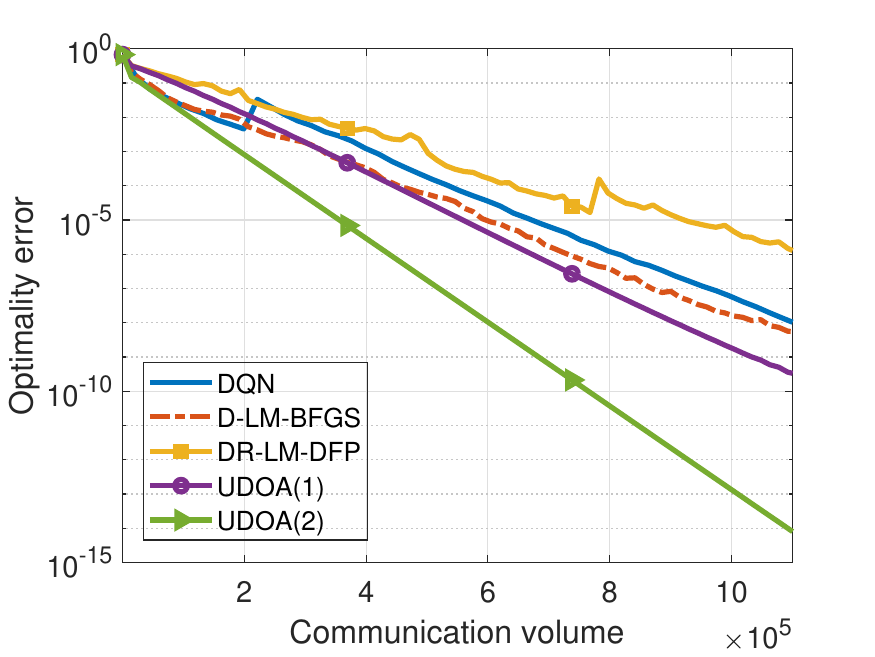}%
		\label{fig_d_com}}
	\caption{Optimality error of quasi-Newton decentralized algorithms for minimizing the nonconvex logistic regression problem \eqref{nonconvex_logistic_problem} on different datasets against communication volume.}
	\label{quasi_com}
\end{figure*}

\begin{figure*}[!t]
	\centering
	\subfloat[]{\includegraphics[width=2.5in]{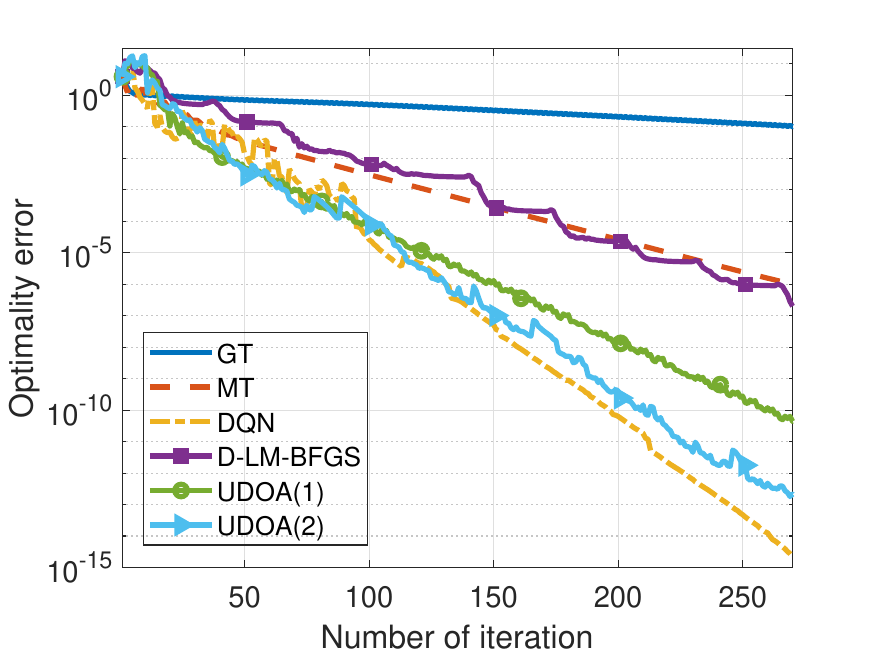}%
		\label{fig1_iter}}
	\subfloat[]{\includegraphics[width=2.5in]{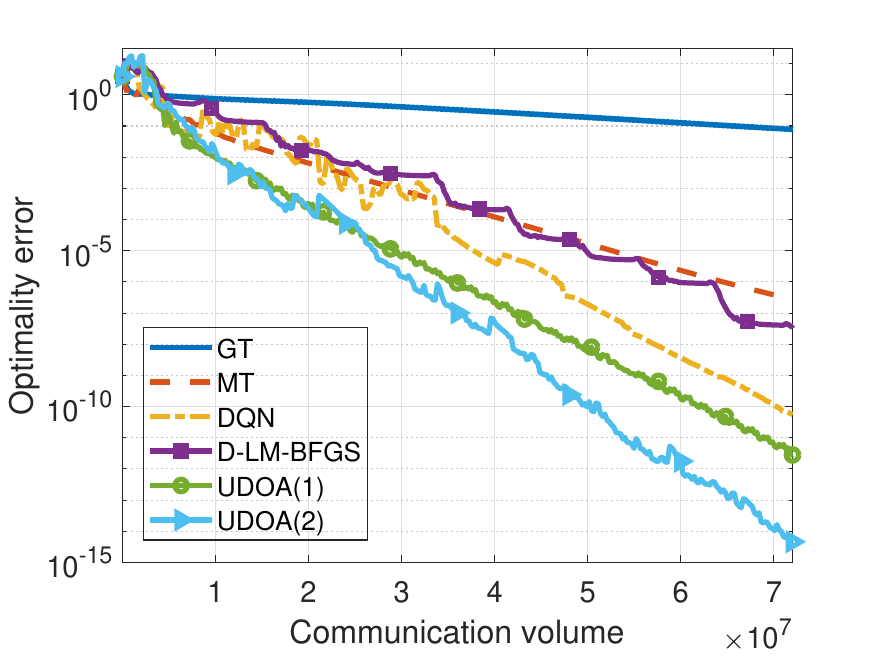}%
		\label{fig1_com}}
	\hfil
	\subfloat[]{\includegraphics[width=2.5in]{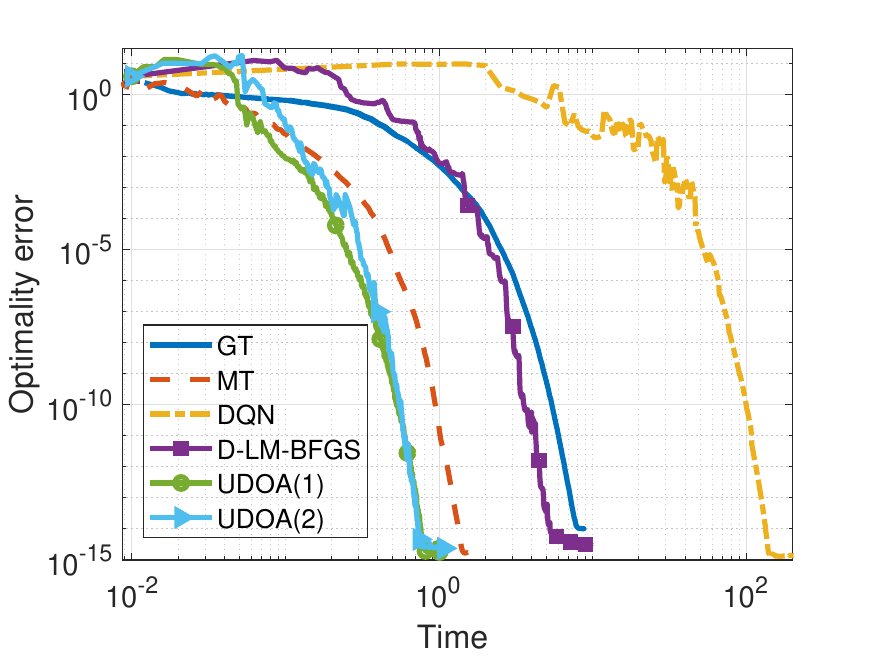}%
		\label{fig1_time}}
	\caption{Optimality error of quasi-Newton decentralized algorithms for minimizing the nonconvex logistic regression problem \eqref{nonconvex_logistic_problem} on \textbf{colon-cancer} datasets.}
	\label{colon}
\end{figure*}

\begin{figure*}[!t]
	\centering
	\subfloat[]{\includegraphics[width=2.5in]{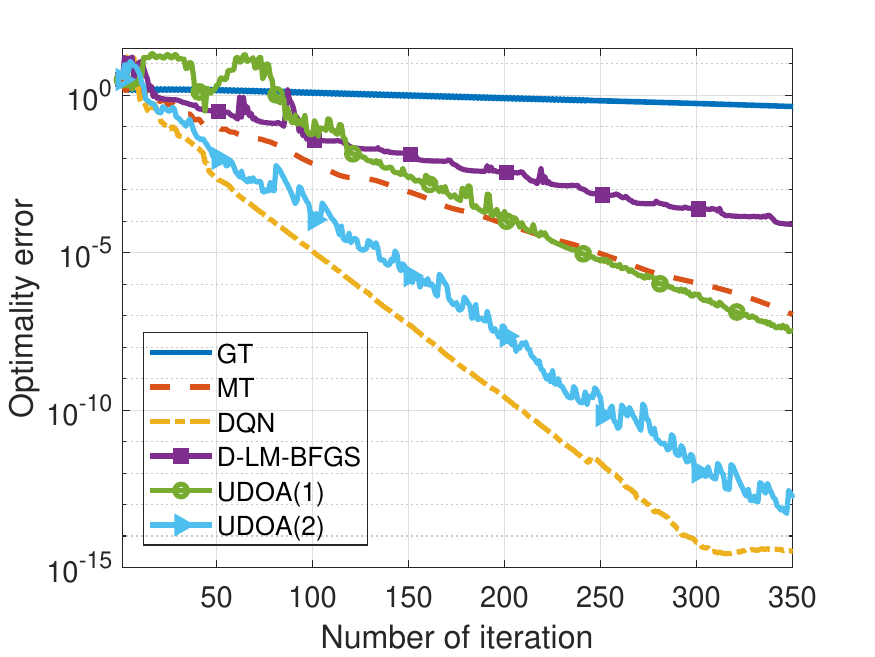}%
		\label{fig_iter}}
	\subfloat[]{\includegraphics[width=2.5in]{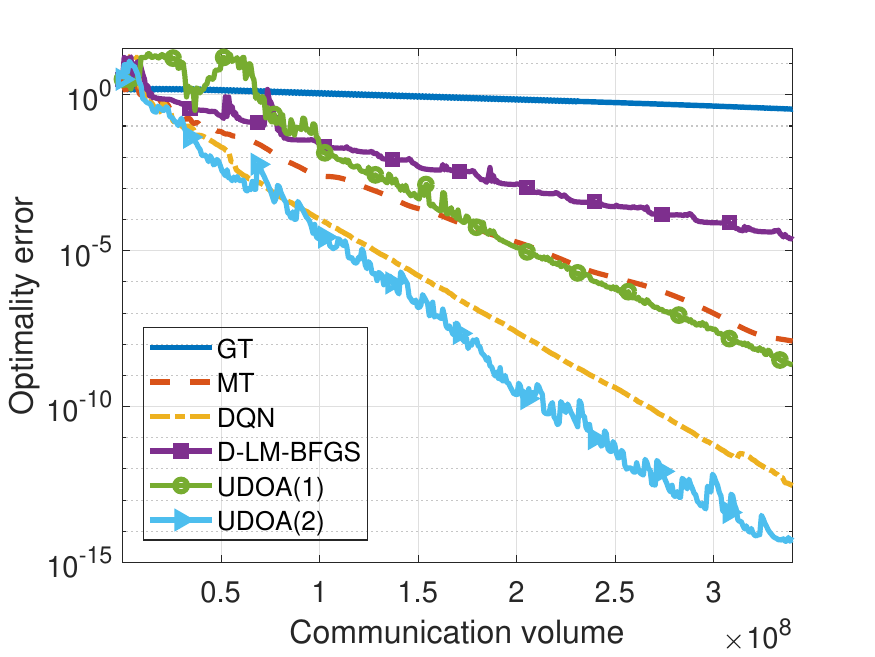}%
		\label{fig_com}}
	\hfil
	\subfloat[]{\includegraphics[width=2.5in]{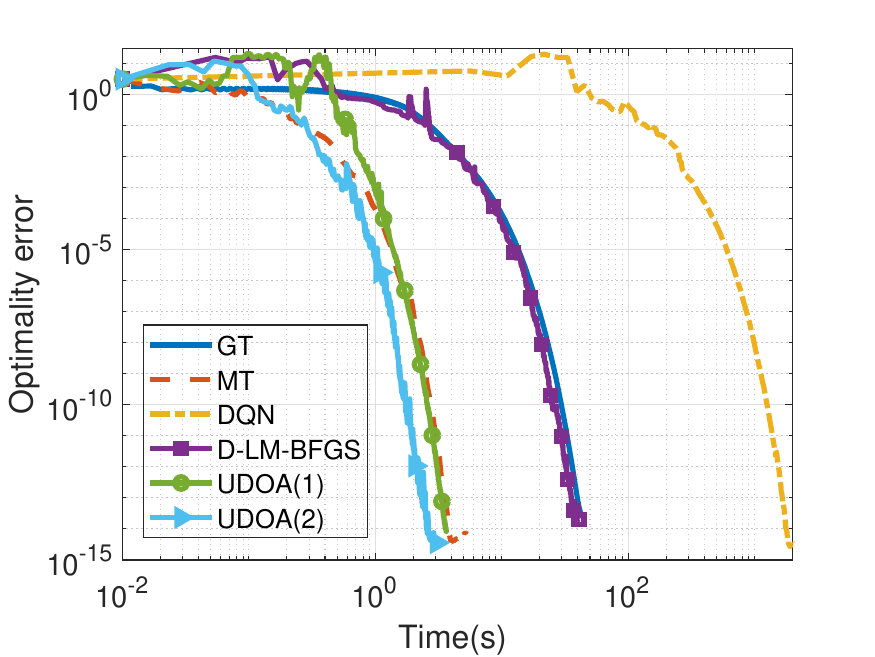}%
		\label{fig_time}}
	\caption{Comparisons with quasi-Newton algorithms for minimizing the nonconvex logistic regression problem \eqref{nonconvex_logistic_problem} on \textbf{duke breast-cancer} datasets.}
	\label{breast}
\end{figure*}

\section{CONCLUSIONS}
In this paper, we investigate the unification and generalization of various important algorithms for solving nonconvex decentralized optimization,
 where local nodes collaboratively minimize the sum of their local nonconvex objectives. We provide a unified algorithmic framework that includes a wide range of decentralized 
gradient tracking methods and quasi-Newton methods. Within a unified analytical structure, 
under proper choices of the stepsizes, we show a subsequence of the iterates converges to stationary point. 
Under the K{\L} property of the aggregated objective function, 
we establish the entire iterate sequence converges with a specific non-asymptotic rate.
Moreover,  several new quasi-Newton decentralized algorithms within our unified algorithmic framework
are proposed. 
Our  numerical experiments on solving the nonconvex decentralized binary classification problem indicate that the newly proposed quasi-Newton algorithms exhibit generally superior performance versus other
well-established comparison algorithms, such as GT, GUT, DSMT, MT, DQN, D-LM-BFGS, and DR-LM-DFP.


\bibliographystyle{amsplain}
\bibliography{refs}

\appendix
\section{Analytical tools}
\begin{lemma}\label{young}
	(Young's inequality) For any two vectors $\m{v}_1,\m{v}_2\in \mathbb{R}^p$, 
	\begin{align*}
		2\m{v}_1\tr\m{v}_2 \leq& \eta \|\m{v}_1\|^2+ \frac{1}{\eta}\|\m{v}_2\|^2,\\
		\|\m{v}_1+\m{v}_2\|^2 \leq& (1+\eta) \|\m{v}_1\|^2+ \left(1+\frac{1}{\eta}\right)\|\m{v}_2\|^2.
	\end{align*}
\end{lemma}
\begin{lemma}\label{important}
	(Jensen's inequality) For any set of vectors $\{\m{v}_i\}_{i=1}^n \subset \mathbb{R}^p$,
	$$\left\|\frac{1}{n}\sum_{i=1}^n\m{v}_i\right\|^2\leq\frac{1}{n}\sum_{i=1}^n\|\m{v}_i\|^2;$$ 
	and for any set of matrices $\{\m{N}_i\}_{i=1}^n \subset \mathbb{R}^{n\times p}$,
	$$\left\|\frac{1}{n}\sum_{i=1}^n\m{N}_i\right\|^2\leq\frac{1}{n}\sum_{i=1}^n\|\m{N}_i\|^2.$$ 
\end{lemma}
\begin{lemma}\label{important2}
	\cite[Lemma 5]{chen2024enhancing} 
	For $0 \leq a_2 \leq a_1$ and $\theta<1$, we have
	$$
	a_1-a_2 \leq \frac{1}{1-\theta} a_1^{\theta} \left(a_1^{1-\theta}-a_2^{1-\theta}\right).
	$$ 
\end{lemma}

\begin{lemma}\label{important3}
	\cite[Theorem 2]{attouch2009convergence} 
	For the nonnegative sequence $\{a^t\}$ satisfying
	$$
	(a^{t})^{\frac{\theta}{1-\theta}} \leq c (a^{t-1}-a^{t}),
	$$
	where $\theta \in (1/2,1)$ and $c>0$, there exists $c'>0$ such that
	$$
	a^t \leq c' t^{-\frac{1-\theta}{2\theta-1}}.
	$$
\end{lemma}

\section{Algorithmic procedure for UDOAs}

\begin{algorithm}[htb]
	\caption{The memoryless SR1 method -- UDOA(1)}
	\label{UDOA1}
	
	\SetKwInOut{Input}{Input}
	\SetKwInOut{Output}{Output}
	\SetKw{Parallel}{in parallel}
	
	\Input{Initial point $\m{x}_i^0,~i\in \mathcal{V}=\{1,\ldots,n\}$, Maximum iteration $T>0$, Stepsize $\alpha>0$, Mixing matrix $\tilde{\m{W}}$, Parameters $u \gg l > 0$ (with $l \leq 1 \leq u$).}
	\Output{$\{\m{x}_i^T\}_{i=1}^n$.}
	
	\BlankLine
	Initialize $t = 0$ and $\m{d}_i^{0} = -\m{v}_i^{0} = -\m{g}_i^{0}$ for all $i \in \mathcal{V}$\;
	
	\While{$t < T$}{
		\For{each node $i \in \mathcal{V}$ \Parallel}{
			$\m{x}_i^{t+1} = \sum_{j \in \mathcal{N}_i}\tilde{W}_{ij} (\m{x}_j^{t}+\alpha \m{d}_j^{t})$\;
			$\m{v}_i^{t+1} = \sum_{j \in \mathcal{N}_i}\tilde{W}_{ij}(\m{v}_j^{t}+\m{g}_j^{t+1}-\m{g}_j^{t})$\;
			
			$\m{s}_i^{t} = \m{x}_i^{t+1}-\m{x}_i^t$, $\check{\m{y}}_i^{t} = \m{v}_i^{t+1}-\m{v}_i^t$, and $ \m{w}_i^t = \m{s}_i^{t} - \check{\m{y}}_i^{t}$\;
			
			\eIf{$ (\m{w}_i^t)\tr \check{\m{y}}_i^{t} \neq 0$ \bf{and} $\left(1 + \frac{\|\m{w}_i^t\|^2}{ (\m{w}_i^t)\tr \check{\m{y}}_i^{t}} \right) \in [{l},{u}]$}{
				$\m{d}_i^{t+1} = -\m{v}_i^{t+1}-\frac{(\m{w}_i^t)\tr\m{v}_i^{t+1}}{ (\m{w}_i^t)\tr \check{\m{y}}_i^{t}} \m{w}_i^t$\;
			}{
				$\m{d}_i^{t+1} = -\m{v}_i^{t+1}$\;
			}
		}
		$t = t+1$\;
	}
\end{algorithm}

\begin{algorithm}[htb]
	\caption{The memoryless BFGS method -- UDOA(2)}
	\label{UDOA2}
	
	\SetKwInOut{Input}{Input}
	\SetKwInOut{Output}{Output}
	\SetKw{Parallel}{in parallel}
	
	\Input{Initial point $\m{x}_i^0,~i\in \mathcal{V}=\{1,\ldots,n\}$, Maximum iteration $T>0$, Stepsize $\alpha>0$, Mixing matrix $\tilde{\m{W}}$, Parameters $u \gg l >0, \varrho>0$.}
	\Output{$\{\m{x}_i^T\}_{i=1}^n$.}
	
	\BlankLine
	Initialize $t = 0$ and $\m{d}_i^{0} = -\m{v}_i^{0} = -\m{g}_i^{0}$ for all $i \in \mathcal{V}$\;
	
	\While{$t < T$}{
		\For{each node $i \in \mathcal{V}$ \Parallel}{
			
			$\m{x}_i^{t+1} = \sum_{j \in \mathcal{N}_i}\tilde{W}_{ij} (\m{x}_j^{t}+\alpha \m{d}_j^{t})$\;
			$\m{v}_i^{t+1} = \sum_{j \in \mathcal{N}_i}\tilde{W}_{ij}(\m{v}_j^{t}+\m{g}_j^{t+1}-\m{g}_j^{t})$\;
			
			$\m{s}_i^{t} = \m{x}_i^{t+1}-\m{x}_i^t$ and $\check{\m{y}}_i^{t} = \m{v}_i^{t+1}-\m{v}_i^t$\;

			\eIf{$(\m{s}_i^{t})\tr \check{\m{y}}_i^t > 0$ \bf{and} $[\lambda_{\min } (H_i^t(\check{\m{y}}_i^t)),\lambda_{\max } (H_i^t(\check{\m{y}}_i^t))] \subset [l,u]$}{
				$\m{y}_i^t = \check{\m{y}}_i^t$\;
			}{
				$h_i^t = \varrho+\max\left\{-\frac{(\m{s}_i^t)\tr(\m{g}_i^{t+1}-\m{g}_i^t)}{\|\m{s}_i^t\|^2},0\right\}$\;
				$\m{y}_i^t = \m{g}_i^{t+1}-\m{g}_i^t+h_i^t\m{s}_i^t$\;
			}

			$\m{d}_i^{t+1} = -\frac{(\m{s}_i^t)\tr \m{y}_i^t}{\|\m{y}_i^t\|^2} \m{v}_i^{t+1} +
			\left(\frac{(\m{y}_i^t)\tr\m{v}_i^{t+1}}{\|\m{y}_i^t\|^2} -2 \frac{(\m{s}_i^t)\tr\m{v}_i^{t+1}}{(\m{s}_i^{t})\tr\m{y}_i^t}\right)\m{s}_i^t +
			\frac{(\m{s}_i^t)\tr\m{v}_i^{t+1}}{\|\m{y}_i^t\|^2}\m{y}_i^t$\;
		}
		$t = t+1$\;
	}
\end{algorithm}

\begin{algorithm}[htb]
	\caption{The corrected quasi-Newton method -- UDOA(3), UDOA(4)}
	\label{UDOA3}
	
	\SetKwInOut{Input}{Input}
	\SetKwInOut{Output}{Output}
	\SetKw{Parallel}{in parallel}
	
	\Input{Initial point $\m{x}_i^0,~i\in \mathcal{V}=\{1,\ldots,n\}$, Maximum iteration $T>0$, Stepsize $\alpha>0$, Mixing matrix $\tilde{\m{W}}$, Parameters $1>\lambda>0, \hat{L}>0, \tau \in \{1,2\}$.}
	\Output{$\{\m{x}_i^T\}_{i=1}^n$.}
	
	\BlankLine
	Initialize $t = 0$ and $\m{d}_i^{0} = -\m{v}_i^{0} = -\m{g}_i^{0}$ for all $i\in \mathcal{V}$\;
	
	\While{$t < T$}{
		\For{each node $i \in \mathcal{V}$ \Parallel}{
			$\m{x}_i^{t+1} = \sum_{j \in \mathcal{N}_i}\tilde{W}_{ij} (\m{x}_j^{t}+\alpha \m{d}_j^{t})$\;
			$\m{v}_i^{t+1} = \sum_{j \in \mathcal{N}_i}\tilde{W}_{ij}(\m{v}_j^{t}+\m{g}_j^{t+1}-\m{g}_j^{t})$\;
			$\m{s}_i^{t} = \m{x}_i^{t+1}-\m{x}_i^t$ and $\check{\m{y}}_i^{t} = \m{v}_i^{t+1}-\m{v}_i^t$\;
			
			\eIf{$(\m{s}^{t}_{i})\tr\check{\m{y}}^{t}_i \leq \lambda\|\m{s}^{t}_{i}\|^2$}{
				$\hat{\eta}_i^t = \frac{(1-\lambda)\|\m{s}^{t}_{i}\|^2}{\|\m{s}^{t}_{i}\|^2-(\m{s}^{t}_{i})\tr\check{\m{y}}^{t}_i}$\;
			}{
				$\hat{\eta}_i^t = 1$\;
			}
			
			$\eta_i^t = \min\left\{\hat{\eta}_i^t,\frac{\hat{L}\|\m{s}^{t}_{i}\|}{\|\check{\m{y}}^{t}_i\|}\right\}$\;
			$\breve{\m{y}}^{t}_i = \eta_i^t\check{\m{y}}^{t}_i+(1-\eta_i^t)\m{s}^{t}_{i}$\;

			$p^{t}_i = \frac{\|\breve{\m{y}}^{t}_i\|^2}{(\m{s}^{t}_{i})\tr\breve{\m{y}}^{t}_i}$, 
			$\m{z}^{t}_i = \breve{\m{y}}^{t}_i-\tau p^{t}_i\m{s}^{t}_i$\;
			
			$\m{d}_i^{t+1} = -\m{v}_i^{t+1}+\frac{(\m{z}^{t}_i)\tr\m{v}_i^{t+1}}{2(\m{s}^{t}_{i})\tr\breve{\m{y}}^{t}_i}\m{s}^{t}_{i}+\frac{(\m{s}^{t}_{i})\tr\m{v}_i^{t+1}}{2(\m{s}^{t}_{i})\tr\breve{\m{y}}^{t}_i}\m{z}^{t}_i$\;
		}
		$t = t+1$\;
	}
\end{algorithm}
\end{document}